\documentclass[11pt,reqno]{amsart}
\usepackage{mathrsfs}
\usepackage{hyperref}

\usepackage{amsfonts}
\usepackage{amsmath}
\usepackage{amssymb}
\usepackage{multicol}
\usepackage{stmaryrd}
\usepackage{cite}
\usepackage{epsfig}
\usepackage{color}
\usepackage{graphics}
\usepackage{graphicx}
\usepackage{epstopdf}
\usepackage{multicol,graphics}
\newcommand\bes{\begin{eqnarray}}
\newcommand\ees{\end{eqnarray}}

\newtheorem{theorem}{Theorem}[section]
\newtheorem{lemma}[theorem]{Lemma}
\newtheorem{corollary}[theorem]{Corollary}

\newtheorem{remark}[theorem]{Remark}
\newtheorem{proposition}[theorem]{Proposition}

\numberwithin{equation}{section}
\allowdisplaybreaks

\allowdisplaybreaks[4]
\newcommand{\rd}{\mathrm{d}}
\newcommand{\lf}{\left}
\newcommand{\rr}{\right}
\begin{document}

\title[nonlocal epidemic model with a new free boundary condition]{Dynamics of a nonlocal epidemic model with a new free boundary condition, part 1: Spreading-vanishing dichotomy}

\author[Y. Chen, Y. Du, W.-T. Li and R. Wang]{Yao Chen$^{1}$, Yihong Du$^{2}$, Wan-Tong Li$^{1}$ and Rong Wang$^{1, *}$}
\thanks{\hspace{-.6cm}
$^1$ School of Mathematics and Statistics, Lanzhou University, Lanzhou, Gansu 730000, P. R. China.}
\thanks{\hspace{-.6cm}
$^2$ School of Science and Technology, University of New England, Armidale, NSW 2351, Australia.}
\thanks{\hspace{-.6cm}
$^*$ {Corresponding author} (wrong@lzu.edu.cn)}

\date{\today}

\begin{abstract}
This paper investigates the long-time dynamics of a nonlocal epidemic model with free boundaries, where a pathogen with density $u(t,x)$ and the infected humans with density $v(t,x)$ evolve according to a reaction-diffusion system with nonlocal diffusion over a one dimensional interval $[g(t), h(t)]$, which represents the epidemic region expanding through its boundaries $x=g(t)$ and $x=h(t)$, known as free boundaries. Such a model  with free boundary conditions based on those of Cao et al. \cite{fb27} was considered by several works. Inspired by recent works
 of Feng et al. \cite{fb20} and Long et al. \cite{fb5}, we propose a new free boundary condition, where the expansion rate of the epidemic region, determined by $h'(t)$ and $g'(t)$,   is proportional to a linear combination of the outward flux of the pathogen \(u\) through the range boundary (as in \cite{fb27}) and the weighted total population of infected individuals \(v\) within the region (as in \cite{fb5}). We prove that the system under this new free boundary condition is well-posed, and its long-time dynamical behavior is characterized by a spreading-vanishing dichotomy. Moreover, we obtain sharp criteria for this dichotomy, including a sharp threshold in terms of the initial data $(u_0,v_0)$; and  by studying a related eigenvalue problem, we also find a sharp threshold in terms of the diffusion rate, which complements  related results in Nguyen and Vo \cite{fb7}. This is Part $1$ of a two part series. In Part $2$, we will determine the spreading speed of the model when spreading occurs, and  for some typical classes of kernel functions, we will obtain the precise rates of accelerated spreading.

\vspace{1em}
\textbf{Keywords}: Nonlocal diffusion; Epidemic model; Free boundary

\textbf{2020 Mathematics Subject Classification}: 35K57, 35R35, 92D30

\end{abstract}

\maketitle

\tableofcontents

\section{Introduction}\label{Int}
\noindent

To better understand the spreading of epidemic diseases, various mathematical models have been developed to capture key features of disease transmission dynamics. Among them, one important class focuses on modeling interactions between human populations and their environment. A classical example is the cholera transmission model proposed by Capasso and Paveri-Fontana \cite{fb13} in 1979, which describes the faecal-oral transmission mechanism and takes the form:
\begin{equation}\label{1}
u^{\prime}(t)=-a u(t)+e v(t), \quad v^{\prime}(t)=-b v(t)+G(u(t)), \quad t>0,
\end{equation}
where \( a \), \( b \), and \( e \) are positive constants. Here, \( u(t) \) and \( v(t) \) represent the average densities of the pathogen and the infected human population at time \( t \), respectively. The quantity \( 1/a \) denotes the average lifetime of the pathogen in the environment, and \( 1/b \) is the average infection period of infected individuals. The coefficient \( e \) reflects the multiplicative effect by which the infected human population increases the density of the pathogen. The function \( G(u) \) characterizes the infection rate of humans and is assumed to satisfy
\begin{itemize}
\item[\bf(G1):] $G \in C^{1}([0, \infty]), G(0)=0, G^{\prime}(z)>0$ for all $z \geq 0$;
\item[\bf(G2):] $\left(\frac{G(z)}{z}\right)^{\prime}<0$ for $z>0$ and $\lim _{z \rightarrow+\infty} \frac{G(z)}{z}<\frac{a b}{e}$.
\end{itemize}
A simple example of such a function is $G(z)=\frac{\alpha z}{1+z^\lambda}$ with $\alpha \in(0, a b / e), \ \lambda\in(0,1]$.
It was proved in \cite{fb13} that the basic reproduction number
\begin{equation}\label{C}
\mathcal{R}_0:=\frac{e G^{\prime}(0)}{a b}
\end{equation}
is a key determinant for the dynamics of \eqref{1} (with positive initial value $(u(0),v(0))$). More specifically, if $\mathcal{R}_{0}\leq 1$, then the epidemic eventually dies out and $(u(t), v(t)) \rightarrow (0,0)$ as $t\rightarrow \infty$. If $\mathcal{R}_{0}>1$, then the epidemic persists and the solution converges to the unique positive equilibrium $(u^{*}, v^{*})$,  determined by
\begin{equation}\label{q}
\frac{G\left(u^{*}\right)}{u^{*}}=\frac{a b}{e} \quad \text { and } \quad v^{*}=\frac{a}{e} u^{*}.
\end{equation}
Subsequent works \cite{fb14,fb15,fb16,fb17} extended model \eqref{1} to include spatial movement over a bounded spatial region and investigated threshold parameters that characterize persistence or extinction. The corresponding Cauchy problem (with the entire Euclidean space as the spatial domain of the system) has been studied in \cite{fb18,fb19}, where traveling waves of the system are used to estimate the spreading speed of the epidemic.

During the early phase of an epidemic, disease transmission typically occurs within a limited region. As individuals move, the infected area gradually expands. The boundary of this region is usually called the epidemic's propagation front. However, the  Cauchy problem models in \cite{fb18, fb19} are not able  to provide the precise location of this front. To overcome this limitation, a diffusion model with free boundaries, inspired by the work of Du and Lin \cite{fb21}, has been proposed and investigated in several recent works, which  has the following form:
\begin{equation}\label{2}
\begin{cases}
u_{t}=d_1 u_{x x}-a u+e v, & t>0, x \in(g(t), h(t)),\\
v_{t}=d_2 v_{x x}-b v+G(u), & t>0, x \in(g(t), h(t)), \\
u(t,x)=v(t,x)=0, & t>0, x \in\{g(t), h(t)\}, \\
h^{\prime}(t)=-\mu [u_{x}(t,h(t))+\rho v_{x}(t,h(t))], & t>0, \\
g^{\prime}(t)=-\mu [u_{x}(t,g(t))+\rho v_{x}(t,g(t))], & t>0, \\
h(0)=-g(0)=h_{0}, & \\ u(x, 0)=u_{0}(x), v(x, 0)=v_{0}(x), & x \in\left[-h_{0}, h_{0}\right],
\end{cases}
\end{equation}
where $h(t)$ and $g(t)$ are the moving boundaries of the infected region $[g(t), h(t)]$, and  $d_1>0,d_2\geq 0, \mu>0,\rho\geq 0$ are constants. The interval $\left[-h_{0}, h_{0}\right]$ represents the initial population range of the species. The equations for $h^{\prime}(t)$ and $g^{\prime}(t)$ mean that the expanding rate of the infected region is proportional to the spatial gradient of $u$ and $v$ at the front, which can be deduced from reasonable biological assumptions as in \cite{BDK}.

 Problem \eqref{2} was first considered by Ahn et al. \cite{fb22} for the special case $d_2=\rho=0$. They proved that the long-time dynamical behavior of \eqref{2} exhibits a spreading-vanishing dichotomy: either $(g(t), h(t))$ remains within a bounded subset of $\mathbb{R}$ for all $t > 0$, and $(u(t, x), v(t, x))$ converges uniformly to $(0, 0)$ as $t \to \infty$ (the vanishing case); or $(g(t), h(t))$ tends to $(-\infty,+\infty)$ as $t \to \infty$, and $(u(t, x), v(t, x))$ converges to a unique positive steady state (the spreading case).
Moreover, when spreading occurs in problem \eqref{2}, Zhao et al. \cite{fb23} showed that there exists a uniquely determined constant $c_0 > 0$ such that
$$
\lim _{t \rightarrow \infty} \frac{h(t)}{t}=\lim _{t \rightarrow \infty} \frac{-g(t)}{t}=c_{0}.
$$
The model \eqref{2} has been further extended to include nonlocal interaction mechanisms \cite{fb24} and time delay effects \cite{fb26} under the assumption $d_2 = \rho = 0$, as well as the general case with $d_2 > 0$ and $\rho \geq 0$ \cite{fb25}.

Recently, Cao et al. \cite{fb27} first studied a Fisher-KPP model with nonlocal diffusion and free boundaries, and since then, such problems have attracted considerable attention, leading to many interesting  works (see, for example, \cite{fb38, fb39, fb40}). In particular, a nonlocal version of model \eqref{2} has been considered, which takes the following form:
{\small\begin{equation}\label{pp}
\begin{cases}
u_{t}=d_{1}\displaystyle\int_{g(t)}^{h(t)}J_{1}(x-y)u(t,y)\mathrm{d}y-d_{1}u-au+ev, & t>0,x\in(g(t),h(t)),\\
v_{t}=d_{2}\displaystyle\int_{g(t)}^{h(t)}J_{2}(x-y)v(t,y)\mathrm{d}y-d_{2}v-bv+G(u), & t>0,x\in(g(t),h(t)),\\
u(t,x)=v(t,x)=0, & t>0,x\in\{g(t),h(t)\},\\
h^{\prime}(t)=\displaystyle\int_{g(t)}^{h(t)}\displaystyle\int_{h(t)}^{\infty}[\mu J_{1}(x-y)u(t,x)+\mu\rho J_{2}(x-y)v(t,x)]\mathrm{d}y\mathrm{d}x,
& t>0,\\
g^{\prime}(t)=-\displaystyle\int_{g(t)}^{h(t)}\displaystyle\int_{-\infty}^{g(t)}[\mu J_{1}(x-y)u(t,x)+\mu\rho J_{2}(x-y)v(t,x)]\mathrm{d}y\mathrm{d}x,
& t>0,\\
h(0)=-g(0)=h_{0},u(0,x)=u_{0}(x),v(0,x)=v_{0}(x),& |x|\leq h_{0},
\end{cases}
\end{equation}}
where $x=g(t)$ and $x=h(t)$ are the moving boundaries to be determined together with $u(t, x)$ and $v(t, x)$, which are always assumed to be identically $0$ for $x \in \mathbb{R} \backslash[g(t), h(t)]$.

The kernel function $J_i (x)~(i=1,2)$ is assumed to satisfy the following condition
\begin{itemize}
\item[\bf(J):]$~J_i\in C\left(\mathbb{R}\right)\cap L^\infty\left(\mathbb{R}\right),~ J_i(x)=J_i(-x)\geq 0,~ J_i(0)>0,~\displaystyle\int_{\mathbb{R}} J_i( x) \mathrm{d}x= 1,~i=1,2.$
\end{itemize}
Problem \eqref{pp} has the advantage over \eqref{2} in that it may capture the spread of epidemics with long-distance dispersal.  The free boundary conditions in \eqref{pp} can be biologically explained as in \cite{fb27}.

When $d_2 = \rho = 0$, Zhao et al. \cite{fb28} proved that the long-time behaviour of \eqref{pp}, similar to that of \eqref{2}, exhibits a spreading-vanishing dichotomy, and Du and Ni \cite{fb2} further determined the corresponding spreading speed in the spreading case. A striking difference between \eqref{pp} and \eqref{2} was identified in \cite{fb2}, showing that the spreading in \eqref{pp} may have infinite asymptotic speed (known as accelerated spreading) depending on whether the kernel functions satisfy the following condition:
\begin{itemize}
\item[\bf (J1):] $ \displaystyle\int_{0}^{\infty} x J_i(x) \mathrm{d} x<+\infty \mbox{ for $i=1$ or 2}.$
\end{itemize}
More precisely, they proved that
$$
\lim _{t \rightarrow \infty} \frac{h(t)}{t}=-\lim _{t \rightarrow \infty} \frac{g(t)}{t}=
\begin{cases}
c_{0} \in(0, \infty) & \text { if }\textbf{(J1)} \text { holds}, \\ \infty & \text { if }\textbf{(J1)} ~\text {does not hold}.
\end{cases}
$$
Subsequently, for the case $d_2 > 0$ and $\rho \geq 0$, Chang and Du \cite{fb4} studied the long-time dynamics of problem \eqref{pp} and determined the spreading speed of the moving fronts, along with the rate of accelerated spreading when $J_i(x) \sim |x|^{-\gamma}$ for $|x| \gg 1$ and $\gamma \in (1,2]$. When the term $ev$ is replaced by a more general function $H(v)$, Nguyen and Vo \cite{fb7} examined the impact of diffusion and established a sharp threshold for $d = d_1 = d_2$. Further developments have focused on the incorporation of nonlocal interaction terms, where the term $ev$ is replaced by  $e \int_{g(t)}^{h(t)} K(x - y) v(t, y)\mathrm{d}y$. In such a case, when $d_2 = \rho = 0$, the model was considered in \cite{fb29, fb30}, while the more general case $d_2 > 0$, $\rho \geq 0$ was studied in \cite{fb31, fb32, fb33}. Recently, Li et al \cite{fb34, fb35} further investigated a nonlocal modification of $G(u)$, replacing it with $G\left(\int_{g(t)}^{h(t)} K(x - y) u(t, y)\mathrm{d}y\right)$.

The entire space version of problem \eqref{pp}, which does not involve any free boundaries, takes the form
\begin{equation}\label{z}
\begin{cases}
u_t=d_1 \displaystyle\int_{\mathbb{R}} J_1(x-y) u(t,y) \mathrm{d}y-d_1 u-a u+e v, & t>0, x \in \mathbb{R}, \\
v_t=d_2 \displaystyle\int_{\mathbb{R}} J_2(x-y) v(t,y) \mathrm{d} y-d_2 v-b v+G(u), & t>0, x \in \mathbb{R}, \\ u(0,x)=u_0(x), v(0,x)=v_0(x), & x \in \mathbb{R},
\end{cases}
\end{equation}
which has been investigated by many authors, and a small sample can be found in  \cite{fb2,fb9,fb36} and the references therein.

To explain the new free boundary condition to be used in this paper, let us  note that if we define
\[
W_J(x):=\int_{x}^{+\infty} J(y) {d} y,
\]
 then the  free boundary condition in \eqref{pp} can be written equivalently as
 \begin{equation}\begin{cases}
\displaystyle h^{\prime}(t) =\mu \int_{g(t)}^{h(t)} \big[u(t, x)   W_{J_1}(h(t)-x)+\rho v(t,x)W_{J_2}(h(t)-x)\big] {d} x,\\
\displaystyle g'(t)=-\mu  \int_{g(t)}^{h(t)} \big[u(t, x)   W_{J_1}(h(t)-x)+\rho v(t,x)W_{J_2}(h(t)-x)\big] {d} x.
\end{cases}
\label{hw}
\end{equation}

Recently Feng et al. \cite{fb20} revisited the locally diffusive Fisher-KPP free boundary model of Du and Lin \cite{fb21}, where the Stefan type free boundary condition $h'(t)=-\mu u_x(t, h(t))$ in \cite{fb21} was replaced by
\begin{equation}\label{4}
h^{\prime}(t)=\mu \int_0^{h(t)} u(t, x) W(h(t)-x) \mathrm{d} x,
\end{equation}
and the weight function $W(x)$ is sign-changing and resembles those used in bird flight models (see, e.g., \cite{fb37}).
 This condition is based on the biological  assumption that the range boundary's movement is determined by the weighted total population within the range, independent of the species' diffusion strategy. More recently,
 Long et al. \cite{fb5} applied a similar free boundary condition to the nonlocal diffusion model of Cao et al. \cite{fb27}, with  $W(x)$ nonnegative and locally Lipschitz continuous in $[0,\infty)$, $W(0)>0$, but independent of the dispersal kernel governing the spatial movement of the species. Their work shows that the dynamics of the new model are mostly similar to those of the old model in \cite{fb27}, while some new propagation phenomena do occur.

Motivated by these works, in this paper, we consider the following variation of \eqref{pp}:
{\small \begin{equation}\label{A}
\begin{cases}
u_{t}=d_{1} \displaystyle\int_{g(t)}^{h(t)} J_{1}(x-y) u(t,y) \mathrm{d} y-d_{1} u-a u+e v, & t>0, x \in(g(t), h(t)), \\
v_{t}=d_{2} \displaystyle\int_{g(t)}^{h(t)} J_{2}(x-y) v(t,y) \mathrm{d} y-d_{2} v-b v+G(u), & t>0, x \in(g(t), h(t)), \\
u(t,x)=v(t,x)=0, & t>0, x \in\{g(t), h(t)\}, \\
h^{\prime}(t)=\mu\displaystyle\int_{g(t)}^{h(t)} \big[ u(t,x)W_{J_1}(h(t)-x)+ \rho v(t,x)W(h(t)-x)\big] \mathrm{d} x, & t>0, \\
g^{\prime}(t)=-\mu\displaystyle\int_{g(t)}^{h(t)}\big[ u(t,x)W_{J_1}(h(t)-x)+ \rho v(t,x)W(h(t)-x)\big] \mathrm{d} x, & t>0, \\
h(0)=-g(0)=h_{0}, u(0,x)=u_{0}(x), v(0,x)=v_{0}(x), & |x| \leq h_{0},
\end{cases}
\end{equation}}
where the parameters $d_1,d_2,a,b,e,\mu,\rho,h_0$ are given positive constants, and $W$ satisfies
\begin{itemize}
\item[\bf(W):] $W\in L^\infty([0,\infty))$ is nonnegative and locally Lipschitz  in $[0,\infty)$ with $W(0)>0$.
\end{itemize}

Throughout this paper, we assume that the kernel function $J_i (x)~(i=1,2)$ satisfies {\bf(J)} and $W$ satisfies {\bf(W)}.
The initial functions $u_0,v_0$ satisfy
\begin{equation}\label{B}
u_0,v_0\in C([-h_0,h_0]),~ u(\pm h_0)=v(\pm h_0)=0,~ u_0(x),v_0(x)>0~\text{in}~(-h_0,h_0),
\end{equation}
and the function $G(u)$ satisfies {\bf(G1)} and {\bf(G2)}. Moreover, we always assume
\begin{equation}\label{H}
u(t,x)=0,~v(t,x)=0~\text{for}~t\geq0,~x\notin (g(t),h(t)).
\end{equation}

The new free boundary conditions in \eqref{A} seem reasonable from a biological point of view,  where the growth rate of $h(t)$ is proportional to a linear combination of
\[
\displaystyle \int_{g(t)}^{h(t)}  u(t,x)W_{J_1}(h(t)-x)\rd x \mbox{ and } \int_{g(t)}^{h(t)} v(t,x)W(h(t)-x) \mathrm{d} x,
\]
with the former representing the total population of pathogens (e.g., cholera bacteria originating from human waste) that move out of the infected region $[g(t), h(t)]$ across the right boundary $h(t)$ via environmental media such as water, but unlike the situation of the pathogens,  it is plausible that the contribution to the advance of $h(t)$ from the infected humans is determined by the total population of infected individuals in the region, as in \cite{fb5}, expressed by the second term above,
with a weight function $W(x)$ independent of the dispersal kernel $J_2(x)$. So in \eqref{A} the expansion rate of $h(t)$ is proportional to a linear combination of the outward flux of the pathogen \( u \) through the range boundary and the weighted total population of infected individuals \( v \) in the infected region.

Our main results on \eqref{A} are the following theorems.

\begin{theorem}\label{T11}
$($Global existence and uniqueness$)$. Suppose that {\bf{(J)}} and {\bf(W)} hold, $G(u)$ satisfies {\bf(G1)-(G2)}, and $\left(u_0(x), v_0(x)\right)$ satisfies \eqref{B}. Then problem \eqref{A} admits a unique  solution $(u, v, g, h)$ defined for all $t > 0$.
\end{theorem}

\begin{theorem}\label{T12}
$($Spreading-vanishing dichotomy$)$. Under the conditions of Theorem \ref{T11}, suppose that $(u, v, g, h)$ is the solution of \eqref{A} and $\mathcal R_0\not=1$. Then one of the following cases must occur:
\begin{itemize}
\item[(i)] {\bf Vanishing:}
 $$\begin{cases}
 \mbox{$\lim \limits_{t \rightarrow
    \infty}(g(t),h(t))=(g_\infty,h_\infty)$ is a finite interval,}\\
\lim _{t \rightarrow \infty}(u(t,x), v(t,x))=(0,0) \text { uniformly for } x \in[g(t), h(t)].
\end{cases}$$

\item[(ii)] {\bf Spreading:}
$$\begin{cases}
\lim \limits_{t \rightarrow \infty} h(t)=-\lim \limits_{t \rightarrow \infty} g(t)=\infty,\\
\lim _{t \rightarrow \infty}(u(t,x), v(t,x))=\left(u^{*}, v^{*}\right) \text { locally uniformly for } x \in \mathbb{R}.
\end{cases}
$$
\end{itemize}
\end{theorem}

\begin{theorem}\label{T13}
$($Spreading-vanishing criteria$)$. Under the conditions of Theorem \ref{T11}. Let $(u, v, g, h)$ be the solution of \eqref{A}. Then the following conclusions hold:
\begin{itemize}
\item[(i)] If $\mathcal{R}_0 < 1$, then  vanishing always happens.
\item[(ii)] If $\mathcal{R}_0 \geq\left(1+\frac{d_1}{a}\right)\left(1+\frac{d_2}{b}\right)$, then spreading always happens.

\item[(iii)] If $1<\mathcal{R}_0<\left(1+\frac{d_1}{a}\right)\left(1+\frac{d_2}{b}\right)$, then there exists a unique ${L}^*>0$ such that
  \begin{itemize}
    \item[(a)] if $h_0 \geq {L}^*$, then spreading always happens;
    \item[(b)] if $0<h_0<{L}^*$, then there exists $\mu^*>0$ depending on $(u_0, v_0)$ such that  vanishing happens when $0<\mu \leq \mu^*$ and spreading happens when $\mu>\mu^*$;
     \item[(c)] if $0<h_0<{L}^*$, $(u_0,v_0)=\sigma(\psi_1,\psi_2)$ with $\sigma>0$ and $(\psi_1,\psi_2)$ satisfying \eqref{B}, and either $J_1(x)>0$ in $\mathbb{R}$ or $W(x)>0$ in $[0,2L^{*}]$, then there exists a unique $\sigma^*> 0$ such that vanishing occurs when $0<\sigma\leq\sigma^*$ and spreading occurs when $\sigma > \sigma ^*$.

  \end{itemize}
\end{itemize}
\end{theorem}

The constant ${L}^*$ depends only on $\left(a, b, e, G^{\prime}(0), d_1, d_2, J_1, J_2\right)$, which is determined by an eigenvalue problem.

Our next result is about the sharp criteria for the spreading-vanishing phenomena in terms of the   diffusion rate $(d_1, d_2)=(sd_1^0, sd_2^0)$. More precisely, when $\mathcal{R}_0 > 1$, by a related principal eigenvalue problem, there exists a threshold value $s=d^*$ such that spreading always occurs if $s \in\left(0, d^*\right]$ while both spreading and vanishing  may happen if $s\in (d^*,+\infty)$.

\begin{theorem}\label{T14}
Under the conditions of Theorem \ref{T11}, suppose that $(d_1, d_2)=(sd_1^0, sd_2^0)$, $\mathcal{R}_0 > 1$ and $(u, v, g, h)$ is the solution of \eqref{A}. Then there exists a unique $d^*>0$ such that the following conclusions hold:
\begin{itemize}
\item[(i)] If $0 < s \leq d^*$, then spreading always happens.
\item[(ii)] If $s > d^*$, then there exists $\mu_0^* > 0$ depending on $(u_0,v_0)$ such that  vanishing occurs when $0 < \mu \leq \mu_0^*$ and spreading occurs when $\mu > \mu_0^*$.

\end{itemize}
\end{theorem}

The constant $d^*$ depends only on $\left(a, b, e, G^{\prime}(0), J_1, J_2, h_0\right)$, which is determined by an eigenvalue problem.

\begin{remark}
\textnormal{(i) Theorem \ref{T14} complements the results in section 5 of \cite{fb7}, where $d_1^0=d_2^0$ was assumed and the case $s=d^*$ was left open. }

\textnormal{(ii) Combining Theorems \ref{T13} and \ref{T14}, we can easily deduce that \( d^{*} \geq \frac{\sqrt{(a-b)^2+4eG'(0)}-(a+b)}{2} \).}

\textnormal{(ii)} \textnormal{Theorem \ref{T14} indicates that maintaining a slow diffusion rate is conducive to the persistence of species $u$ and $v$. This result reinforces the findings in several related works of Lou et al.  \cite{fb10,fb11,fb12}, where
  a slow diffusion rate is  always favored. }
\end{remark}

When $\mathcal R_0=1$, it follows from \cite{fb31}  that  vanishing always happens for \eqref{A} if $W=W_{J_2}$. For \eqref{A} with a general $W$, in this case we can only show
\begin{equation}\label{conv-0}
(u(t,x), v(t,x))\to (0, 0) \mbox{ uniformly for $x\in [g(t), h(t)]$ as $t\to\infty$.}
\end{equation}
We can prove $[g_\infty, h_\infty]$ is a finite interval only under some extra conditions. Let us note that when $\mathcal R_0=1$, we may write
\begin{equation}\label{R0=1}
G(u)=\frac {ab}eu-g(u)u, \mbox{ with } g(0)=0,\ g'(u)>0 \mbox{ for } u>0.
\end{equation}

\begin{theorem}\label{thm6}
Suppose $\mathcal R_0=1$, and $(u,g,h)$ is the unique solution of \eqref{A}. Then \eqref{conv-0} holds. Moreover, $[g_\infty, h_\infty]$ is a finite interval
if  either {\rm (i)} or {\rm (ii)} of the following holds:
\begin{itemize}
\item[(i)] $W(x)\leq \eta W_{J_2}(x)$ for all $x\geq 0$ and some constant $\eta>0$,
\item[(ii)] $g(u)$ in \eqref{R0=1} satisfies $g(u)\geq \sigma u^\lambda$ for some $\sigma>0,\ \lambda\in (0,1)$ and all small $u>0$.
\end{itemize}
\end{theorem}

The rest of this paper is organized as follows. In Section 2, we prove the global existence and uniqueness of the solution to \eqref{A}, and present some basic results on comparison principles, on an eigenvalue problem and on some fixed boundary problems. Section 3 is devoted to the proof of Theorems \ref{T12}, \ref{T13}, \ref{T14} and \ref{thm6}.

\section{Global existence and uniqueness, and some other basic results}
\subsection{Global existence and uniqueness}

For convenience, we introduce the following notations. For any given $h_0, T>0$, we define
$$
\begin{aligned}
& \mathbb{H}_{T}=\mathbb{H}_{h_0, T}:=\left\{h \in C([0, T]): h(0)=h_0, h(t) \text { is strictly increasing }\right\}, \\
& \mathbb{G}_{T}=\mathbb{G}_{h_0, T}:=\left\{g \in C([0, T]):-g \in \mathbb{H}_{h_0, T}\right\} .
\end{aligned}
$$

For any $g \in \mathbb{G}_{T}, h \in \mathbb{H}_{T}$, and initial data $(u_{0}, v_{0})$ satisfying \eqref{B}, we define
\[D_T=D_{g, h}^T  :=\{(t, x)\in \mathbb{R}^2: 0<t\leq T, g(t)<x<h(t)\},\]
and
$$
\begin{aligned}
\mathbb{X}_T=\mathbb{X}_{u_0,v_0,g, h}^T  :=\{(\phi_1, \phi_2)\in [C\left(\bar{D}_T\right)]^2:
&(\phi_1(0, x), \phi_2(0, x))=(u_{0}(x), v_0(x)) ~ \mbox{ in }~[-h_0, h_0],\\
&\hspace{-0.9cm}\phi_i \geqslant 0,~\phi_i(t, g(t))=\phi_i(t, h(t))=0 ~ \mbox{ in }[0, T], i=1,2\}.
\end{aligned}
$$

\begin{lemma}\label{vv}
\textnormal{(Maximum principle)}. Suppose that {\bf{(J)}} holds, and $g \in \mathbb{G}_{T}, h \in \mathbb{H}_{T}$ for some $h_0,T>0$. Assume that $u(t, x), v(t, x),u_t(t, x)$ as well as $v_t(t, x)$ are continuous in $\bar{D}_T$ and satisfy
\begin{equation*}
\begin{cases}
u_t(t, x) \geq d_1 \displaystyle\int_{g(t)}^{h(t)} J_1(x-y) u(t, y) \rd y-d_1 u+a_{11}u+a_{12} v, & 0<t\leq T,x \in(g(t), h(t)), \\
v_t(t, x) \geq d_2 \displaystyle\int_{g(t)}^{h(t)} J_2(x-y) v(t, y) \rd y-d_2 v+a_{21}v+a_{22} u, & 0<t\leq T,x \in(g(t), h(t)), \\
u(t, x) \geq 0, ~v(t, x) \geq 0, & 0<t\leq T, x\in \{g(t),h(t)\}, \\
u(0, x) \geq 0, ~ v(0, x) \geq 0, & |x| \leq h_0 ,
\end{cases}
\end{equation*}
where $a_{11},a_{12},a_{21},a_{22} \in L^\infty (D_T)$ with $a_{12}, a_{22}\geq 0$. Then $u(t, x)\geq 0,~ v(t, x)\geq 0$ for all $(t, x) \in \bar{D}_T$. Moreover, if $u(0, x)\not \equiv 0, v(0, x)\not \equiv 0$ for $|x| \leq h_0$, then $u(t, x)>0, v(t, x)>0$ in $D_T$.
\end{lemma}
\noindent{\bf{Proof.}}  The lemma can be proved by the comparison arguments similar to the proof of Lemma 3.1 in \cite{fb1}. Here we omit the details.\qed

\begin{lemma}\label{cc}
Assume that {\bf{(J)}} holds, $G(u)$ satisfies {\bf(G1)-(G2)}, $h_0 > 0$ and $(u_0, v_0)$ satisfies \eqref{B}. Then for any $T > 0$ and $(g, h) \in \mathbb{G}_T \times \mathbb{H}_T$, the problem
\begin{equation}\label{a1}
\begin{cases}
u_t = d_1 \displaystyle\int_{g(t)}^{h(t)} J_1(x-y) u(t,y) \, \rd y - d_1u - au + ev, & 0 < t \leq T, \, x \in (g(t), h(t)), \\
v_t = d_2 \displaystyle\int_{g(t)}^{h(t)} J_2(x-y) v(t,y) \, \rd y - d_2v - bv + G(u), & 0 < t \leq T, \, x \in (g(t), h(t)), \\
u(t,x) = v(t,x) = 0, & 0 < t \leq T, \, x \in \{g(t), h(t)\}, \\
-g(0) = h(0) = h_0, \quad u(0,x) = u_0(x), \quad v(0,x) = v_0(x), & |x| \leq h_0
\end{cases}
\end{equation}
admits a unique solution $(U_{g,h}, V_{g,h}) \in \mathbb{X}_T$. Furthermore, $U_{g,h}, V_{g,h}$ satisfy
\begin{align}\label{a6}
&0 < U_{g,h}(t,x) \leq A := \max\left\{ u^*, \|u_0\|_\infty, \frac{e}{a} \|v_0\|_\infty \right\}, \\
&0 < V_{g,h}(t,x) \leq B := \max\left\{ \|v_0\|_\infty, \frac{G(A)}{b} \right\}
\end{align}
for all $t \in (0,T]$ and $x \in (g(t), h(t))$, where $u^*$ is uniquely determined by \eqref{q} if $\mathcal{R}_0 > 1$, and $u^* = 0$ if $\mathcal{R}_0 \leq 1$.
\end{lemma}
\noindent{\bf{Proof.}} The proof of this lemma is similar to the previous works (see, e.g., \cite{fb4} or \cite[Theorem 4.1]{fb1}. Here, we omit the details of the proof. \qed

Let $(U_{g,h}, V_{g,h})$ be the unique positive solution of \eqref{a1} given by Lemma \ref{cc}. Now, using such a $(U_{g,h}, V_{g,h})$, we define the mapping $\mathcal{F}$ by
\[
\mathcal{F}(g,h)(t) = (\tilde{g}(t), \tilde{h}(t)),
\]
where
\begin{align}\label{a2}
\tilde{h}(t): =
&\ h_0 + \mu\int_0^t  \int_{g(\tau)}^{h(\tau)}  U_{g,h}(\tau,x)W_{J_1}(h(\tau)-x)  \rd x \rd \tau\nonumber\\
&+\mu\rho \int_0^t\int_{g(\tau)}^{h(\tau)} V_{g,h}(\tau,x) W(h(\tau)-x)\rd x \rd \tau, \\
\tilde{g}(t) =
&-h_0 - \mu \int_0^t \int_{g(\tau)}^{h(\tau)} U_{g,h}(\tau,x)W_{J_1}(x-g(\tau)) \rd x \rd \tau\nonumber\\
&-\mu\rho \int_0^t \int_{g(\tau)}^{h(\tau)} V_{g,h}(\tau,x) W(x-g(\tau)) \rd x \rd \tau
\end{align}
for $(g,h) \in \mathbb{G}_T \times \mathbb{H}_T$ and $0 < t \leq T$.

For any positive constants $\varepsilon, s, \alpha, \beta$, we define
\begin{equation*}\begin{aligned}
\Sigma_{\varepsilon, s, \alpha, \beta} := \Big\{& (g,h) \in \mathbb{G}_s \times \mathbb{H}_s : \ h(t) - g(t) \leq 2h_0 + \frac{\varepsilon}{4}, \ \mbox{ and }
\\
&\sup_{0 \leq t_1 < t_2 \leq s} \frac{g(t_2) - g(t_1)}{t_2 - t_1} \leq -\mu\alpha,\ \inf_{0 \leq t_1 < t_2 \leq s} \frac{h(t_2) - h(t_1)}{t_2 - t_1} \geq \mu\beta
 \text{ for } t \in [0,s] \Big\}
\end{aligned}
\end{equation*}
and for $(g_1, h_1), (g_2, h_2) \in \Sigma_{\varepsilon, s, \alpha, \beta}$,
\[
d = d((g_1, h_1), (g_2, h_2)) := \|g_1 - g_2\|_{C([0,s])} + \|h_1 - h_2\|_{C([0,s])}.
\]
 Clearly, $(\Sigma_{\varepsilon, s, \alpha, \beta}, d)$ is a complete metric space.

Next, we show that $\mathcal{F}$ maps a suitable closed subset of $\mathbb{G}_T \times \mathbb{H}_T$ into itself, and is a contraction mapping.

\begin{lemma}\label{a3}  There exist positive constants $\varepsilon_0, T_0, \alpha_0, \beta_0$ such that
\begin{align*}
\mathcal{F}(\Sigma_{\varepsilon_0, s, \alpha_0, \beta_0}) \subseteq \Sigma_{\varepsilon_0, s, \alpha_0, \beta_0} \quad \text{for } s \in (0, T_0].
\end{align*}
\end{lemma}

\noindent{\bf{Proof.}} By the conditions {\bf{(J)}} and {\bf{(W)}}, there exist constants $\varepsilon_0 \in (0, h_0/4)$ and $\delta_0 > 0$ such that
\begin{align}\label{a4}
J_1(x-y) \geq \delta_0 \quad \text{for } |x-y| \leq \varepsilon_0, \quad \text{and} \quad \int_{\frac{5}{4}\varepsilon_0 }^{2h_0 - \varepsilon_0} W(x)\rd x > 0.
\end{align}
We denote
\begin{equation}\label{m}
\begin{aligned}
T_0 :&= \frac{\varepsilon_0}{4\mu(2h_0 + \varepsilon_0) \left[ A + 2\rho B \sup\limits_{x \in [0, 2h_0 + \varepsilon_0]} W(x) \right]},\\
\beta_0 :&= \frac{\varepsilon_0}{4} \delta_0 e^{-(d_1 + a) T_0} \int_{h_0-\frac{\varepsilon_0}{4} }^{h_0 } u_0(x) \rd x + \rho e^{-(d_2 + b) T_0} \min_{|x| \leq h_0 - \varepsilon_0} v_0(x) \int_{\frac{5}{4}\varepsilon_0}^{2h_0 - \varepsilon_0} W(x) \rd x,\\
\alpha_0 :&= \frac{\varepsilon_0}{4} \delta_0 e^{-(d_1 + a) T_0} \int_{-h_0}^{-h_0 + \frac{\varepsilon_0}{4}} u_0(x) \rd x + \rho e^{-(d_2 + b) T_0} \min_{|x| \leq h_0 - \varepsilon_0} v_0(x) \int_{\frac{5}{4}\varepsilon_0}^{2h_0 - \varepsilon_0} W(x) \rd x.
\end{aligned}
\end{equation}
Fix $s \in (0, T_0]$ and $(g,h) \in \Sigma_{\varepsilon_0, s, \alpha_0, \beta_0}$. By the definitions of $\tilde{h}$ and $\tilde{g}$, we have  $(\tilde{g}, \tilde{h}) \in C^1([0,s]) \times C^1([0,s])$ with
\[
\tilde{g}(t) < -h_0 = \tilde{g}(0) \mbox{ and }\tilde{h}(t) > h_0 = \tilde{h}(0)\mbox{ for }t \in (0,s].
\]
Furthermore, according to assumptions {\bf{(J)}} and {\bf{(W)}}, we deduce that
\[
\tilde{g}'(t) < 0 \quad \text{and} \quad \tilde{h}'(t) > 0 \quad \text{for} \quad t \in (0,s].
\]
Hence, $(\tilde{g}, \tilde{h})$ belongs to $\mathbb{G}_s \times \mathbb{H}_s$. Let $(U_{g,h}, V_{g,h})$ be a positive solution of \eqref{a1}. Then we have
\begin{equation*}
\begin{cases}
(U_{g,h})_t \geq -(d_1 + a) U_{g,h}, & 0 < t \leq s, \, x \in (g(t), h(t)), \\
(V_{g,h})_t \geq -(d_2 + b) V_{g,h}, & 0 < t \leq s, \, x \in (g(t), h(t)), \\
U_{g,h}(t,x) = V_{g,h}(t,x) = 0, & 0 < t \leq s, \, x \in \{g(t), h(t)\}, \\
U_{g,h}(0,x) = u_0(x), \quad V_{g,h}(0,x) = v_0(x), & |x| \leq h_0,
\end{cases}
\end{equation*}
which implies that
\begin{align}\label{a5}
U_{g,h}(t,x) &\geq e^{-(d_1 + a)t}u_0(x) \geq e^{-(d_1 + a)s} u_0(x) \quad \text{for} \quad t \in (0,s], \, x \in [-h_0, h_0], \nonumber\\
V_{g,h}(t,x) &\geq e^{-(d_2 + b)t} v_0(x) \geq e^{-(d_2 + b)s}v_0(x) \quad \text{for} \quad t \in (0,s], \, x \in [-h_0, h_0].
\end{align}
By $(g,h) \in \Sigma_{\varepsilon_0, s, \alpha_0, \beta_0}$, we have $h(s) - g(s) \leq 2h_0 + \frac{\varepsilon_0}{4}$ and
\[
h(t) \in \left[ h_0, h_0 + \frac{\varepsilon_0}{4} \right], \quad g(t) \in \left[ -h_0 - \frac{\varepsilon_0}{4}, -h_0 \right] \quad \text{for} \quad t \in [0,s].
\]
Using \eqref{a6} and \eqref{a2}, we easily see
\begin{align*}
&[\tilde{h}(t) - \tilde{g}(t)]' \\
&~~= \mu \left[ \int_{g(t)}^{h(t)} \int_{h(t)}^{+\infty} J_1(x-y) U_{g,h}(t,x)  \rd y \rd x + \int_{g(t)}^{h(t)} \int_{-\infty}^{g(t)} J_1(x-y) U_{g,h}(t,x)  \rd y  \rd x \right] \\
&~~\quad + \mu \rho \left[ \int_{g(t)}^{h(t)} V_{g,h}(t,x) W(h(t)-x)  \rd x + \int_{g(t)}^{h(t)} V_{g,h}(t,x) W(x-g(t))  \rd x \right] \\
&~~\leq \mu \int_{g(t)}^{h(t)} \int_{-\infty}^{+\infty} J_1(x-y) U_{g,h}(t,x)  \rd y  \rd x \\
&~~ \ \ \ \ + \mu \rho  \int_{0}^{h(t)-g(t)} \big[V_{g,h}(t, h(t)-y) +V_{g,h}(t, y+g(t))\big] W(y)  \rd y  \\
&~~\leq \mu A [h(t) - g(t)] + 2 \mu \rho B \sup_{x \in [0, 2h_0 + \varepsilon_0]} W(x) [h(t) - g(t)] \quad \text{for } t \in [0, s].
\end{align*}
This, together with the definition of $T_0$, allows us to derive
\begin{align*}
\tilde{h}(t) - \tilde{g}(t)& \leq 2h_0 + T_0 \mu A (2h_0 + \varepsilon_0) + 2 \mu \rho B T_0\sup_{x \in [0, 2h_0 + \varepsilon_0]} W(x) (2h_0 + \varepsilon_0)\\
&\leq 2h_0 + T_0 \left[ \mu (2h_0 + \varepsilon_0) \left( A + 2 \rho B \sup_{x \in [0, 2h_0 + \varepsilon_0]} W(x) \right) \right]\\
&= 2h_0 + \frac{\varepsilon_0}{4} \quad \text{for } t \in [0, s].
\end{align*}
Combining this with \eqref{B}, \eqref{a4} and \eqref{a5}, we obtain
\begin{align*}
&\mu \int_{g(t)}^{h(t)} \int_{h(t)}^{+\infty} J_1(x-y) U_{g,h}(t,x)  \rd y \rd x + \mu \rho \int_{g(t)}^{h(t)} V_{g,h}(t,x) W(h(t)-x)  \rd x\\
&\geq \mu \int_{h(t)-\frac{\varepsilon_0}{2}}^{h(t)} \int_{h(t)}^{h(t)+\frac{\varepsilon_0}{2}} J_1(x-y) U_{g,h}(t,x) \rd y  \rd x + \mu \rho e^{-(d_2+b)T_0} \int_{-h_0}^{h_0} W(h(t)-x) v_0(x)  \rd x\\
&\geq \mu \int_{h_0-\frac{\varepsilon_0}{4}}^{h_0} \int_{h_0+\frac{\varepsilon_0}{4}}^{h_0+\frac{\varepsilon_0}{2}} J_1(x-y) U_{g,h}(t,x) \rd y  \rd x \\
&\quad + \mu \rho e^{-(d_2+b)T_0} \min_{|x| \leq h_0-\varepsilon_0} v_0(x) \int_{-h_0+\varepsilon_0}^{h_0-\varepsilon_0} W(h(t)-x)  \rd x\\
&\geq \mu \frac{\varepsilon_0}{4} \delta_0 e^{-(d_1+a)T_0} \int_{h_0-\frac{\varepsilon_0}{4}}^{h_0} u_0(x)\rd x + \mu \rho e^{-(d_2+b)T_0} \min_{|x| \leq h_0-\varepsilon_0} v_0(x) \int_{h(t)-h_0+\varepsilon_0}^{h(t)+h_0-\varepsilon_0} W(x)  \rd x\\
&\geq \mu \left[\frac{\varepsilon_0}{4} \delta_0 e^{-(d_1+a)T_0} \int_{h_0-\frac{\varepsilon_0}{4}}^{h_0} u_0(x) \rd x + \rho e^{-(d_2+b)T_0} \min_{|x| \leq h_0-\varepsilon_0} v_0(x) \int_{\frac{5}{4}\varepsilon_0}^{2h_0-\varepsilon_0} W(x) \rd x\right]\\
&=: \mu \beta_0 > 0 \quad \text{for } t \in [0,s].
\end{align*}
This combined with \eqref{a2} yields
\begin{align}\label{a7}
\tilde{h}'(t) \geq \mu \beta_0 \quad \text{for } t \in [0,s].
\end{align}
Similarly, we obtain
\begin{align}\label{a8}
\tilde{g}'(t) \leq -\mu \alpha_0 \quad \text{for } t \in [0,s]
\end{align}
with $\alpha_{0}$ defined by \eqref{m}. It follows from the above discussions that
\[
\mathcal{F}(\Sigma_{\varepsilon_0, s, \alpha_0, \beta_0}) \subseteq \Sigma_{\varepsilon_0, s, \alpha_0, \beta_0} \quad \text{for } s \in (0, T_0].
\]
\qed

\begin{lemma}\label{a9}
 There exists $T_* \in (0, T_0]$ such that $\mathcal{F}$ is a contraction mapping on $\Sigma_{\varepsilon_0, s, \alpha_0, \beta_0}$ for $s \in (0, T_*]$, where $\varepsilon_0, T_0, \alpha_0, \beta_0$ are defined in Lemma \ref{a3}.
\end{lemma}
\noindent{\bf{Proof.}} For $(g_i, h_i) \in \Sigma_{\varepsilon_0, s, \alpha_0, \beta_0}$, $i = 1, 2$, we define
\[
\Omega_s = D_{g_1, h_1}^s \cup D_{g_2, h_2}^s, \quad U_i = U_{g_i, h_i}, \quad V_i = V_{g_i, h_i}, \quad \mathcal{F}(g_i, h_i) = (\tilde{g}_i, \tilde{h}_i).
\]
We extend $U_i$ and $V_i$ by $0$ in $([0, s] \times \mathbb{R}) \setminus D_{g_i, h_i}^s$. By {\bf{(J)}}, {\bf{(W)}} and \eqref{a2}, we obtain
\begin{align*}
&\left|\tilde{h}_1(t) - \tilde{h}_2(t)\right|\\
&~~\leq \mu \int_0^t \left|\int_{g_1(\tau)}^{h_1(\tau)} \int_{h_1(\tau)}^{+\infty} J_1(x-y) U_1(\tau,x)  \rd y  \rd x + \rho \int_{g_1(\tau)}^{h_1(\tau)} V_1(\tau,x) W(h_1(\tau)-x)  \rd x \right. \\
&~~\left.\quad - \int_{g_2(\tau)}^{h_2(\tau)} \int_{h_2(\tau)}^{+\infty} J_1(x-y) U_2(\tau,x)  \rd y  \rd x  - \rho \int_{g_2(\tau)}^{h_2(\tau)} V_2(\tau,x) W(h_2(\tau)-x) \rd x  \right|\rd\tau \\
&~~\leq \mu \int_0^t\left| \int_{g_1(\tau)}^{h_1(\tau)}\int_{h_1(\tau)}^{+\infty} J_1(x-y) U_1(\tau,x)  \rd y  \rd x  -  \int_{g_2(\tau)}^{h_2(\tau)}\int_{h_2(\tau)}^{+\infty} J_1(x-y) U_2(\tau,x)  \rd y  \rd x \right| \rd \tau \\
&~~\quad + \mu\rho \int_0^t \left|\int_{g_1(\tau)}^{h_1(\tau)} V_1(\tau,x) W(\left|h_1(\tau)-x\right|)  \rd x - \int_{g_2(\tau)}^{h_2(\tau)} V_2(\tau,x) W(\left|h_2(\tau)-x\right|)  \rd x \right|\rd \tau \\
&~~\leq \mu \int_0^t \int_{g_1(\tau)}^{h_1(\tau)} \int_{h_1(\tau)}^{+\infty} J_1(x-y) \left|U_1(\tau,x) - U_2(\tau,x)\right|  \rd y  \rd x  \rd\tau \\
&~~\quad + \mu \int_0^t \left| \left(\int_{g_1(\tau)}^{g_2(\tau)} \int_{h_1(\tau)}^{+\infty} + \int_{h_2(\tau)}^{h_1(\tau)} \int_{h_1(\tau)}^{+\infty}+\int_{g_2(\tau)}^{h_2(\tau)}\int_{h_1(\tau)}^{h_2(\tau)}\right) J_1(x-y) U_2(\tau,x)  \rd y  \rd x \right| \rd \tau \\
&~~\quad + \mu \rho \int_0^t \int_{g_1(\tau)}^{h_1(\tau)} W(\left|h_1(\tau) - x\right|) |V_1(\tau,x) - V_2(\tau,x)| \rd x  \rd\tau \\
&~~\quad + \mu \rho \int_0^t \left| \int_{g_1(\tau)}^{g_2(\tau)} W(\left|h_1(\tau) - x\right|) V_2(\tau,x)  \rd x - \int_{h_1(\tau)}^{h_2(\tau)} W(\left|h_1(\tau) - x\right|) V_2(\tau,x)  \rd x \right|  \rd \tau \\
&~~\quad + \mu \rho \int_0^t \int_{g_2(\tau)}^{h_2(\tau)}\left| W(\left|h_1(\tau) - x\right|)- W(\left|h_2(\tau) - x\right|)\right|V_2(\tau,x)  \rd x  \rd \tau \\
&~~\leq 3h_0 \mu s \|U_1 - U_2\|_{C(\bar{\Omega}_s)} + \mu A s \|g_1 - g_2\|_{C([0,s])} + \mu A S \|h_1 - h_2\|_{C([0,s])} \\
&~~\quad + \mu  s 3h_0 A \|J_1\|_\infty \|h_1 - h_2\|_{C([0,s])}+ 3h_0 s\mu \rho \sup_{x \in [0,3h_0]}  W(x) \|V_1 - V_2\|_{C(\bar{\Omega}_S)} \\
&~~\quad  + s \mu \rho B \sup_{x \in [0,3h_0]} W(x) \|g_1 - g_2\|_{C([0,s])}+ s \mu \rho B \sup_{x \in [0,3h_0]} W(x) \|h_1 - h_2\|_{C([0,s])} \\
&~~\quad  + 3h_0 s \mu \rho B \tilde{L}\|h_1 - h_2\|_{C([0,s])}\\
&~~\leq C_1 s \left[ \|U_1 - U_2\|_{C(\bar{\Omega}_s)} + \|V_1 - V_2\|_{C(\bar{\Omega}_s)} + \|h_1 - h_2\|_{C([0,s])} + \|g_1 - g_2\|_{C([0,s])} \right],
\end{align*}
where $\tilde{L}=\tilde{L} (3h_0)$ is the Lipschitz constant of $W(x)$ on $[0, 3h_0]$ due to the condition {\bf{(W)}}, and $C_1>0$ is a constant depending on $(h_0, A, B, \mu, \rho, J_1, W)$.

Similarly, we have
\[
|\tilde{g}_1(t) - \tilde{g}_2(t)| \leq C_2 s \left[ \|U_1 - U_2\|_{C(\bar{\Omega}_s)} + \|V_1 - V_2\|_{C(\bar{\Omega}_s)} + \|h_1 - h_2\|_{C([0,s])} + \|g_1 - g_2\|_{C([0,s])} \right]
\]
for $t \in [0, s]$ with some $C_2>0$. Hence,
\begin{align}\label{B1}
&\|\tilde{h}_1 - \tilde{h}_2\|_{C([0,s])} + \|\tilde{g}_1 - \tilde{g}_2\|_{C([0,s])}\nonumber\\
&\leq (C_1 + C_2) s \left[ \|U_1 - U_2\|_{C(\bar{\Omega}_s)} + \|V_1 - V_2\|_{C(\bar{\Omega}_s)} + \|h_1 - h_2\|_{C([0,s])} + \|g_1 - g_2\|_{C([0,s])}
\right].
\end{align}

Similar to the second step of the proof of Theorem 2.1 in \cite{fb3} (with only a minor modification), by using \eqref{B1}, we further obtain
$$
\|\tilde{h}_1 - \tilde{h}_2\|_{C([0,s])} + \|\tilde{g}_1 - \tilde{g}_2\|_{C([0,s])} \leq \frac{1}{2} \left[ \|h_1 - h_2\|_{C([0,s])} + \|g_1 - g_2\|_{C([0,s])}\right]
$$
for all $0 < s \leq T_* \leq T_0$. This implies that $\mathcal{F}$ is a contraction mapping on $\Sigma_{\varepsilon_0, s, \alpha_0, \beta_0}$ for every $s \in (0, T_*]$. \qed

\begin{theorem}\label{B2}
Assume that {\bf{(J)}} and {\bf{(W)}} hold, and $G(u)$ satisfies {\bf{(G1)-(G2)}}. Then, for any given $h_0 > 0$ and $(u_0, v_0)$ satisfying \eqref{B1}, problem \eqref{A} has a unique global solution $(u(t,x), v(t,x), g(t), h(t))$. Moreover, for any $T > 0$, we have $(g, h) \in \mathbb{G}_T \times \mathbb{H}_T$ and $(u, v) \in \mathbb{X}_T$.
\end{theorem}

\noindent{\bf{Proof.}} Thanks to the Contraction Mapping Theorem and Lemmas \ref{a3} and \ref{a9}, we know that $\mathcal{F}$ has a unique fixed point $(g_*, h_*)$ in $\Sigma_{\varepsilon_0, s, \alpha_0, \beta_0}$, for every $s \in (0, T_*]$, where $\varepsilon_0, \alpha_0, \beta_0, T_*$  are those given in Lemmas \ref{a3} and \ref{a9}. With this at hand, we can now follow arguments analogous to the third and fourth steps of the proof of Theorem 2.1 in \cite{fb3} to obtain the global existence and uniqueness of the solution $(u, v, g, h)$ to problem \eqref{A}. \qed

\subsection{Some comparison principles}

The following result is well known and follows from a simpler version of the argument used in the proof of Lemma~\ref{vv}.
\begin{lemma}\label{hh}
Suppose that {\bf{(J)}} holds, and $h_0,T>0$. Assume that $u(t, x), v(t, x),u_t(t, x)$ as well as $v_t(t, x)$ are continuous in $\bar{D}_{g_0,h_0}^T$ and satisfy
\begin{equation*}
\begin{cases}
u_t(t, x) \geq d_1 \displaystyle\int_{-h_0}^{h_0} J_1(x-y) u(t, y) \rd y-d_1 u+a_{11}u+a_{12} v, & 0<t\leq T,x \in(-h_0, h_0), \\
v_t(t, x) \geq d_2 \displaystyle\int_{-h_0}^{h_0} J_2(x-y) v(t, y) \rd y-d_2 v+a_{21}v+a_{22} u, & 0<t\leq T,x \in(-h_0, h_0), \\
u(0, x) \geq 0, ~ v(0, x) \geq 0, & |x| \leq h_0 ,
\end{cases}
\end{equation*}
where $a_{11},a_{12},a_{21},a_{22} \in L^\infty (D_{g_0,h_0}^T)$ with $a_{12}, a_{22}\geq 0$. Then $u(t, x)\geq 0,~ v(t, x)\geq 0$ for all $(t, x) \in \bar{D}_{g_0,h_0}^T$. Moreover, if $u(0, x)\not \equiv 0, v(0, x)\not \equiv 0$ for $|x| \leq h_0$, then $u(t, x)>0, v(t, x)>0$ in $(0,T]\times[-h_0,h_0]$.
\end{lemma}

\begin{lemma}\label{mm}
$(\textnormal{Comparison principle})$. Suppose that {\bf{(J)}} and {\bf{(W)}} hold, $G(u)$ satisfies {\bf(G1)-(G2)}, and $(u_0, v_0)$ satisfies \eqref{B}. For $T \in(0,+\infty)$, assume that $\bar{h},\bar{g} \in C([0, T])$ and $ \bar{u}, \bar{v},\bar{u}_t, \bar{v}_t \in C(\bar{D}_{\bar{g},\bar{h}}^T)$ satisfy
$$
\begin{cases}
\bar{u}_t \geq d_1 \displaystyle\int_{\bar{g}(t)}^{\bar{h}(t)} J_1(x-y) \bar{u}(t,y) \rd y-d_1 \bar{u}-a \bar{u}+e \bar{v}, & 0<t\leq T,x \in(\bar{g}(t), \bar{h}(t)), \\
\bar{v}_t \geq d_2 \displaystyle\int_{\bar{g}(t)}^{\bar{h}(t)} J_2(x-y) \bar{v}(t,y) \rd y-d_2 \bar{v}-b \bar{v}+G(\bar{u}), & 0<t\leq T,x \in(\bar{g}(t), \bar{h}(t)), \\
\bar{u}(t, x) \geq 0, ~\bar{v}(t, x) \geq 0, & 0<t\leq T, x\in \{\bar{g}(t),\bar{h}(t)\},\\
\bar{h}^{\prime}(t)\geq\mu\displaystyle\int_{\bar{g}(t)}^{\bar{h}(t)} \displaystyle\int_{\bar{h}(t)}^{\infty} J_{1}(x-y) \bar{u}(t,x)\rd y \rd x\\
\quad \quad \quad \quad+\mu \rho\displaystyle\int_{\bar{g}(t)}^{\bar{h}(t)} \bar{v}(t,x)W(\bar{h}(t)-x) \rd x, & 0<t\leq T, \\
\bar{g}^{\prime}(t)\leq-\mu\displaystyle\int_{\bar{g}(t)}^{\bar{h}(t)} \displaystyle\int_{-\infty}^{\bar{g}(t)} J_{1}(x-y) \bar{u}(t,x)\rd y \rd x\\
\quad \quad \quad \quad-\mu \rho\displaystyle\int_{\bar{g}(t)}^{\bar{h}(t)} \bar{v}(t,x)W(x-\bar{g}(t))\rd x, & 0<t\leq T, \\
\bar{u}(0,x)\geq u_{0}(x), \bar{v}(0,x)\geq v_{0}(x), \bar{h}(0)\geq h_{0}, \bar{g}(0)\leq -h_0, & |x| \leq h_{0},
\end{cases}
$$
then the unique positive solution $(u, v, g, h)$ of \eqref{A} satisfies
\begin{equation}\label{0p}
u(t,x) \leq \bar{u}(t,x), ~ v(t,x) \leq \bar{v}(t,x), ~g(t) \geq \bar{g}(t) ~and ~ h(t)\leq \bar{h}(t)
\end{equation}
for $t \in (0,T] ~and ~x \in [g(t),h(t)]$.
\end{lemma}

\noindent{\bf{Proof.}} By {\bf(G1)}, there exists $\xi=\xi(t,x) \in(0, \bar{u}(t,x)]$ such that $G(\bar{u})=G^{\prime}(\xi) \bar{u}$. According to \eqref{B} and Lemma \ref{vv}, we have $\bar{u}, \bar{v}>0$ for $t \in (0,T]$ and $x \in (\bar{g}(t),\bar{h}(t))$. Hence, $\bar{h}$ and $-\bar{g}$ are strictly increasing.

For small $\epsilon>0$, let $\left(u_\epsilon, v_\epsilon, g_\epsilon, h_\epsilon\right)$ denote the unique solution of \eqref{A} with $\mu$ replaced by $\mu^\epsilon:=\mu(1-\epsilon)$, $h_0$ replaced by $h_0^\epsilon:=h_0(1-\epsilon)$, and $\left(u_0, v_0\right)$ replaced by $\left(u_0^\epsilon, v_0^\epsilon\right)$ satisfying
$$
0\leq u_0^\epsilon(x)<u_0(x),~ 0\leq v_0^\epsilon(x)<v_0(x)~ \text { in }~\left[-h_0^\epsilon, h_0^\epsilon\right],~u_0^\epsilon\left( \pm h_0^\epsilon\right)=v_0^\epsilon\left( \pm h_0^\epsilon\right)=0
$$
and
$$
\left(u_0^\epsilon\left(\frac{h_0}{h_0^\epsilon} x\right), v_0^\epsilon\left(\frac{h_0}{h_0^\epsilon} x\right)\right) \rightarrow\left(u_0(x), v_0(x)\right) ~ \text { as }~ \epsilon \rightarrow 0 ~\text{in}~ (C\left(\left[-h_0, h_0\right]\right))^2.
$$

Now, we claim that $\bar{g}(t)<g_\epsilon(t)<h_\epsilon(t)<\bar{h}(t)$ in $(0, T]$. Otherwise,  there exists $t_1\leq T$ such that
$$
\bar{g}(t)<g_\epsilon(t)<h_\epsilon(t)<\bar{h}(t) ~\text { for }~ t \in\left(0, t_1\right) \text { and }\left[h_\epsilon\left(t_1\right)-\bar{h}\left(t_1\right)\right]\left[g_\epsilon
\left(t_1\right)-\bar{g}\left(t_1\right)\right]=0,
$$
since $\bar{g}(0)<g_\epsilon(0)<h_\epsilon(0)<\bar{h}(0)$ and $h_\varepsilon, g_\varepsilon, \bar{h}, \bar{g} \in C([0,T])$.
Without loss of generality, we assume
$$
h_\epsilon\left(t_1\right)=\bar{h}\left(t_1\right)\text{ and } \bar{g}\left(t_1\right)\leq g_\epsilon\left(t_1\right).
$$
Therefore, we have $\bar{h}^{\prime}(t_1)\leq h_\varepsilon^{\prime}(t_1).$ Meanwhile, it follows from Lemma \ref{vv} that
$$
\bar{u}(t,x)-u_\varepsilon(t,x)>0,~\bar{v}(t,x)-v_\varepsilon(t,x)>0~\text{for}~t \in (0, t_1],x \in (g_\varepsilon(t), h_\varepsilon(t)).
$$
Therefore,
$$
\begin{aligned}
0 \geq  &\bar{h}^{\prime}\left(t_1\right)-h_\epsilon^{\prime}\left(t_1\right) \\
\geq  &\mu\displaystyle\int_{\bar{g}\left(t_1\right)}^{\bar{h}\left(t_1\right)} \displaystyle\int_{\bar{h}\left(t_1\right)}^{\infty} J_1(x-y) \bar{u}\left( t_1,x\right)\rd y\rd x-\mu_\varepsilon\displaystyle\int_{g_\varepsilon(t_1)}^{h_
\varepsilon(t_1)} \displaystyle\int_{h_\varepsilon(t_1)}^{\infty}J_1(x-y) u_\varepsilon\left( t_1,x\right)\rd y \rd x \\
&+\mu\rho\displaystyle\int_{\bar{g}\left(t_1\right)}^{\bar{h}\left(t_1\right)}W(\bar{h}
\left(t_1\right)-x)\bar{v}(t_1,x)\rd x-\mu_\varepsilon\rho\displaystyle\int_{g_\varepsilon
(t_1)}^{h_\varepsilon(t_1)}W(h_\varepsilon(t_1)-x)v_\varepsilon(t_1,x)\rd x\\
>&\mu_\varepsilon\displaystyle\int_{g_\varepsilon(t_1)}^{h_\varepsilon(t_1)} \displaystyle\int_{h_\varepsilon(t_1)}^{\infty}J_1(x-y)(\bar{u}\left( t_1,x\right)-u_\varepsilon\left( t_1,x\right))\rd y \rd x \\ &+\mu_\varepsilon\rho\displaystyle\int_{g_\varepsilon(t_1)}^{h_\varepsilon(t_1)}
W(h_\varepsilon(t_1)-x)(\bar{v}(t_1,x)-v_\varepsilon(t_1,x))\rd x>0,
\end{aligned}
$$
which clearly is a contradiction. Hence, we have
\[
h_\varepsilon(t)<\bar{h}(t)\mbox{ and }\bar{g}(t)<g_\varepsilon(t)\mbox{ for all }t \in (0,T].
\]
It then follows that $\bar{u}(t,x)>u_\varepsilon(t,x)$ and $\bar{v}(t,x)>v_\varepsilon(t,x)$ for all $t \in (0,T]$ and $x \in (g_\varepsilon(t),h_\varepsilon(t))$. Thanks to the continuous dependence of the unique solution of \eqref{A} on the parameters, we can obtain \eqref{0p} by letting $\epsilon \rightarrow 0$.\qed

\begin{remark}\label{909}
\textnormal{
We call \( (\bar{u}, \bar{v}, \bar{g}, \bar{h}) \) in Lemma~\ref{mm} an upper solution of problem~\eqref{A}.
By reversing all the inequalities in Lemma~\ref{mm}, one can similarly define a lower solution and obtain analogous results.
}
\end{remark}

Let \(\left(u^\mu, v^\mu, g^\mu, h^\mu\right)\) denote the solution of problem \eqref{A} to highlight its dependence on the parameter \(\mu\). The following result is a direct consequence of Lemma~\ref{mm}.

\begin{corollary}\label{nn}
Suppose that {\bf{(J)}} and {\bf{(W)}} hold, $G(u)$ satisfies {\bf(G1)-(G2)}. If $\mu_1 \leq \mu_2$, then $h^{\mu_1}(t) \leq h^{\mu_2}(t),~ g^{\mu_1}(t) \geq g^{\mu_2}(t),~u^{\mu_1}(t, x) \leq u^{\mu_2}(t, x)~and~v^{\mu_1}(t, x) \leq v^{\mu_2}(t, x)$ for $t>0$ and $x\in (g^{\mu_1}(t),h^{\mu_1}(t))$.
\end{corollary}

Lemma \ref{mm} has several useful variations, all of which can be proved  by the same  technique. We present one such variation here for future reference.

\begin{lemma}\label{bb}
In Lemma \ref{mm}, assume that $\bar{g}(t) \equiv g(t)$  for $t \in[0, T]$ and $(\bar{u}, \bar{v}, g, \bar{h})$ satisfies
$$
\begin{cases}
\bar{u}_t \geq d_1 \displaystyle\int_{g(t)}^{\bar{h}(t)} J_1(x-y) \bar{u}(t,y) \rd y-d_1 \bar{u}-a \bar{u}+e \bar{v}, & 0<t\leq T,x \in(g(t), \bar{h}(t)), \\
\bar{v}_t \geq d_2 \displaystyle\int_{g(t)}^{\bar{h}(t)} J_2(x-y) \bar{v}(t,y) \rd y-d_2 \bar{v}-b \bar{v}+G(\bar{u}), & 0<t\leq T,x \in(g(t), \bar{h}(t)), \\
\bar{u}(t, x) \geq 0, ~\bar{v}(t, x) \geq 0, & 0<t\leq T, x\in \{g(t),\bar{h}(t)\},\\
\bar{h}^{\prime}(t)\geq\mu\displaystyle\int_{g(t)}^{\bar{h}(t)} \displaystyle\int_{\bar{h}(t)}^{\infty} J_{1}(x-y) \bar{u}(t,x)\rd y \rd x\\
\quad \quad \quad \quad+\mu \rho\displaystyle\int_{g(t)}^{\bar{h}(t)} \bar{v}(t,x)W(\bar{h}(t)-x) \rd x, & 0<t\leq T, \\
 \bar{u}(0,x)\geq u_{0}(x), \bar{v}(0,x)\geq v_{0}(x), \bar{h}(0)\geq h_{0}, & |x| \leq h_{0},
\end{cases}
$$
then
$$
u(t,x) \leq \bar{u}(t,x), ~ v(t,x) \leq \bar{v}(t,x), ~and ~ h(t)\leq \bar{h}(t)~for~ t \in (0,T],~x \in [g(t),h(t)].
$$

\end{lemma}

\subsection{Principal eigenvalue}

For any finite interval $[L_1, L_2 ]$ and constants $d_1, d_2 > 0$, we consider the following eigenvalue problem
\begin{align}\label{c4}
\left\{
\begin{aligned}
&\frac{d_1}{e} \int_{L_1}^{L_2} J_1(x-y) \phi_1(y) \rd y - \frac{d_1}{e} \phi_1 - \frac{a}{e} \phi_1 + \phi_2 + \lambda \phi_1 = 0, & x \in [L_1, L_2], \\
&\frac{d_2}{G'(0)} \int_{L_1}^{L_2} J_2(x-y) \phi_2(y) \rd y -  \frac{d_2}{G'(0)}\phi_2 + \phi_1 - \frac{b}{G'(0)} \phi_2 + \lambda \phi_2 = 0, & x \in [L_1, L_2].
\end{aligned}
\right.
\end{align}
For convenience, we denote
{\small \begin{equation*}\begin{cases}
\mathbf{A}: = \begin{pmatrix}
    -\frac{a}{e} & 1 \\
    1 & -\frac{b}{G(0)}
\end{pmatrix}, \quad
\mathbf{D}: = \begin{pmatrix}
    \frac{d_1}{e} & 0 \\
    0 & \frac{d_2}{G'(0)}
\end{pmatrix}, \\
\textit{\textbf{C}}([L_1, L_2]): = C([L_1, L_2]) \times C([L_1, L_2]),\\
\textit{\textbf{E}}([L_1, L_2]): = L^2([L_1, L_2]) \times L^2([L_1, L_2]).
\end{cases}
\end{equation*}}
Obviously, the constant matrix $\mathbf{A}$ is symmetric, with off-diagonal entries positive. Note that $\textit{\textbf{E}}$ is a Hilbert space equipped with inner product
$$
\langle {\boldsymbol{\phi}}, {\boldsymbol{\psi}}\rangle = \int_{L_1}^{L_2} \phi_1(x) \psi_1(x) \rd x + \int_{L_1}^{L_2} \phi_2(x) \psi_2(x) \rd x \mbox{ for }{\boldsymbol{\phi}} = (\phi_1, \phi_2)^T, {\boldsymbol{\psi}} = (\psi_1, \psi_2)^T.
$$
  We also define $\textit{\textbf{N}}: \textit{\textbf{E}} \rightarrow \textit{\textbf{C}}$ by
\begin{equation*}
(\textit{\textbf{N}}{\boldsymbol{\phi}})(x) = \textnormal{diag }(N_1[\phi_1](x), N_2[\phi_2](x)),
\end{equation*}
where
\[
N_i[\phi_i](x) := \int_{L_1}^{L_2} J_i(x-y) \phi_i(y) \rd y, ~ i = 1, 2.
\]
Then, the eigenvalue problem \eqref{c4} can be rewritten as
\begin{align}\label{c5}
{\boldsymbol{\mathcal{K}}}{\boldsymbol{\phi}} + \lambda {\boldsymbol{\phi}} = \mathbf{0},
\end{align}
with ${\boldsymbol{\mathcal{K}}}: \textit{\textbf{E}} \rightarrow \textit{\textbf{E}}$ defined by ${\boldsymbol{\mathcal{K}}} = \mathbf{D}\textit{\textbf{N}} - \mathbf{D} + \mathbf{A}$.

Next, we define  $\lambda_p({\boldsymbol{\mathcal{K}}})$ and $\lambda_v({\boldsymbol{\mathcal{K}}})$ as follows:
\begin{align}\label{c6}
&\lambda_p({\boldsymbol{\mathcal{K}}}): = \sup \left\{ \lambda \in \mathbb{R} : \exists~ {\boldsymbol\phi} \in \textit{\textbf{C}}([L_1, L_2]), {\boldsymbol\phi} > \mathbf{0},~\text{s.t.}~ {\boldsymbol{\mathcal{K}}}[{\boldsymbol\phi}](x) + \lambda {\boldsymbol\phi}(x) \leq \mathbf{0} \text{ in } [L_1, L_2] \right\},\nonumber\\
&
\lambda_v({\boldsymbol{\mathcal{K}}}): = \inf_{{\boldsymbol\phi} \in \textit{\textbf{E}}([L_1, L_2]), {\boldsymbol\phi} \not\equiv \mathbf{0}}- \frac{\langle{\boldsymbol{\mathcal{K}}}[{\boldsymbol\phi}], {\boldsymbol\phi} \rangle}{\|{\boldsymbol\phi}\|_{\textit{\textbf{E}}([L_1, L_2])}^2}.
\end{align}
Here ${\boldsymbol\phi}=(\phi_1,\phi_2)^T > \mathbf{0}$ in $[L_1, L_2]$ means $\phi_1(x)>0,\ \phi_2(x)>0$ in $[L_1, L_2]$.

The following result follows directly from \cite[Theorem 1]{fb8} and \cite{fb7}.

\begin{proposition}\label{c7}
\textnormal{(\cite[Theorem 1]{fb8})}. Suppose that the kernel $J_i ~(i=1,2)$ satisfies {\bf{(J)}}. Then ${\boldsymbol{\mathcal{K}}}$ is self-adjoint and
\begin{align}\label{c8}
\lambda_p({\boldsymbol{\mathcal{K}}}) = -\sup_{\|{\boldsymbol\phi}\|_{\textit{\textbf{E}}}=1} \langle {\boldsymbol{\mathcal{K}}}{\boldsymbol\phi}, {\boldsymbol\phi} \rangle = \lambda_v({\boldsymbol{\mathcal{K}}}).
\end{align}
\end{proposition}

 The number $\lambda_p=\lambda_p({\boldsymbol{\mathcal{K}}})$ is usually called the principal eigenvalue of the operator ${\boldsymbol{\mathcal{K}}}$, or the principal eigenvalue of \eqref{c4}. We will also denote it by $\lambda_p(L_1, L_2)$ when its dependence on $[L_1, L_2]$ is stressed.
By Proposition \ref{c7} we have
{\small \begin{equation}\label{eigen-define}
\begin{aligned}
\lambda_p= &-\sup_{\|{\boldsymbol\phi}\|_{\textit{\textbf{E}}}=1} \langle {\boldsymbol{\mathcal{K}}}{\boldsymbol\phi}, {\boldsymbol\phi} \rangle\\
=& \inf_{\|{\boldsymbol\phi}\|_{\textit{\textbf{E}}}=1} \left\{ \frac{d_1}{2e} \int_{L_1}^{L_2} \int_{L_1}^{L_2} J_1(x-y) (\phi_1(x) - \phi_1(y))^2 \rd x \rd y \right.\\
&\left. +\frac{d_2}{2G'(0)} \int_{L_1}^{L_2} \int_{L_1}^{L_2} J_2(x-y) (\phi_2(x) - \phi_2(y))^2 \rd x \rd y \right.\\
&- \int_{L_1}^{L_2} \left[ \left( -\frac{a}{e} - \frac{d_1}{e} + \frac{d_1}{e} \int_{L_1}^{L_2} J_1(x-y) \rd y \right) \phi_1^2(x) +2\phi_1(x)\phi_2(x) \right. \\
&\left. \left. + \left( -\frac{b}{G'(0)} - \frac{d_2}{G'(0)} + \frac{d_2}{G'(0)} \int_{L_1}^{L_2} J_2(x-y) \rd y \right) \phi_2^2(x) \right] \rd x \right\}.
\end{aligned}
\end{equation}}

By \cite{fb4} and \cite{fb1}, we can easily deduce the following conclusions.

\begin{proposition}\label{c3}
Suppose that the kernel $J_i ~(i=1,2)$ satisfies {\bf{(J)}}, and $-\infty < L_1 < L_2 < +\infty$. Let $\lambda_p = \lambda_p(L_1, L_2)$ be the principal eigenvalue of \eqref{c4}. Then the following hold true:
\begin{itemize}
    \item[(i)] $\lambda_p(L_1, L_2)$ is strictly decreasing and continuous in $L:=L_2 - L_1$.
    \item[(ii)] If $\mathcal{R}_0 \leq 1$, then $\lambda_p(L_1, L_2) > 0$ for any $L_1<L_2$.
    \item[(iii)] If $\mathcal{R}_0 \geq\left(1+\frac{d_1}{a}\right)\left(1+\frac{d_2}{b}\right)$, then $\lambda_p(L_1, L_2) < 0$ for any $L_1<L_2$.
    \item[(iv)] If $1<\mathcal{R}_0 <\left(1+\frac{d_1}{a}\right)\left(1+\frac{d_2}{b}\right) $, then there exists $L^*>0$ such that
    \[
    \lambda_p(L_1, L_2) = 0 \text{ if } L_2 - L_1 = 2L^*,
    \]
    \[
    \lambda_p(L_1, L_2) > 0 \text{ if } L_2 - L_1 < 2L^*,
    \]
    \[
    \lambda_p(L_1, L_2) < 0 \text{ if } L_2 - L_1 > 2L^*.
    \]
\end{itemize}
\end{proposition}

Next, we investigate the effect of the dispersal rates $(d_1, d_2)$ on the principal eigenvalue of \eqref{c4}, and to stress the dependence on $(d_1, d_2)$, we write $\lambda_p = \lambda_p(d_1, d_2)$. We have the following result.
\begin{proposition}\label{c9}
Suppose that the kernel $J_i ~(i=1,2)$ satisfies {\bf{(J)}}, and  $0<L_1 < L_2 < +\infty$. Let $\lambda_p = \lambda_p(d_1, d_2)$ be the principal eigenvalue of $\eqref{c4}$. Then the following statements are true:
\begin{itemize}
    \item[(i)] $\lambda_p(d_1, d_2)$ is a continuous and strictly increasing function in $(d_1, d_2)$ in the sense that $\lambda_p(d_1, d_2) > \lambda_p(d_1', d_2')$ if $d_i > d_i', i=1,2$.
    \item[(ii)] $\lim\limits_{(d_1, d_2) \to (0,0)} \lambda_p(d_1, d_2) = \displaystyle\frac{1}{2} \left( \frac{a}{e} + \frac{b}{G'(0)} - \sqrt{\left( \frac{a}{e} - \frac{b}{G'(0)} \right)^2 + 4} \right)$.
    \item[(iii)] $\lim\limits_{(d_1, d_2) \to (\infty,\infty)} \lambda_p(d_1, d_2) = +\infty$.

\end{itemize}
\end{proposition}
\noindent{\bf{Proof.}} (i) Let ${\boldsymbol\phi}= (\phi_1, \phi_2)$ be the corresponding positive eigenfunction pair to $\lambda_p(d_1, d_2)$ normalized by $\|{\boldsymbol\phi}\|_{\textit{\textbf{E}}}=1$. Then we have
\begin{equation}\label{d1-d2}\begin{aligned}
\lambda_p(d_1, d_2) = &\ \frac{d_1}{2e} \int_{L_1}^{L_2} \int_{L_1}^{L_2} J_1(x-y)(\phi_1(x) - \phi_1(y))^2 \rd x \rd y \\
&+ \frac{d_2}{2G'(0)} \int_{L_1}^{L_2} \int_{L_1}^{L_2} J_2(x-y)(\phi_2(x) - \phi_2(y))^2 \rd x \rd y\\
&-\int_{L_1}^{L_2}\left[ \left(- \frac{a}{e} +\frac{d_1}{e}\left( \int_{L_1}^{L_2}J_1(x-y)\rd y-1\right)\right) \phi_1^2(x)+2\phi_1(x)\phi_2(x)\right.\\
&\left.+\left(-\frac{b}{G'(0)}+\frac{d_2}{G'(0)}\left(\int_{L_1}^{L_2}J_2(x-y)\rd y
-1\right)\right)\phi_2^2(x) \right]\rd x .
\end{aligned}
\end{equation}
Denote $$L_0:=\max\left\{\frac{2\max\limits_{x\in[L_1,L_2]}\phi_1(x)}{e\min\limits_{x\in[L_1,L_2]}\phi_1(x)
},\frac{2\max\limits_{x\in[L_1,L_2]}\phi_2(x)}{G'(0)\min\limits_{x\in[L_1,L_2]}\phi_2(x)
}\right\}.$$
For any $(d_1, d_2), (d_1', d_2')\in (0,+\infty)\times(0,+\infty)$, a straightforward calculation gives
\begin{align*}
&-\frac{d_1'}{e} \int_{L_1}^{L_2} J_1(x-y) \phi_1(y) \rd y +\frac{d_1'}{e} \phi_1 + \frac{a}{e} \phi_1 - \phi_2\\
&=\lambda_p(d_1,d_2)\phi_1+\frac{1}{e}(d_1-d_1') \int_{L_1}^{L_2} J_1(x-y) \phi_1(y) \rd y-
\frac{1}{e}(d_1-d_1') \phi_1\\
&\geq \lambda_p(d_1,d_2)\phi_1-\frac{1}{e}|d_1-d_1'| \int_{L_1}^{L_2} J_1(x-y) \phi_1(y) \rd y-\frac{1}{e}|d_1-d_1'| \phi_1\\
&\geq\Big(\lambda_p(d_1,d_2)-\frac 2 e\frac{\max\limits_{x\in[L_1,L_2]}\phi_1}
{\min\limits_{x\in[L_1,L_2]}\phi_1}|d_1-d_1'|\Big)\phi_1\\
&\geq\lf[\lambda_p(d_1,d_2)-L_0|d_1-d_1'|\rr]\phi_1.
\end{align*}
Similarly,
\begin{align*}
&-\frac{d_2'}{G'(0)} \int_{L_1}^{L_2} J_2(x-y) \phi_2(y) \rd y +  \frac{d_2'}{G'(0)}\phi_2 - \phi_1 +\frac{b}{G'(0)} \phi_2 \\
&\geq\lf[\lambda_p(d_1,d_2)-L_0|d_2-d_2'|\rr]\phi_2.
\end{align*}
It follows from \cite[Lemma 2.2]{fb1} that
$$
\lambda_p(d_1',d_2')\geq\lambda_p(d_1,d_2)-L_0(|d_1-d_1'|+|d_2-d_2'|).
$$
Similarly, we have
$$
\lambda_p(d_1',d_2')\leq\lambda_p(d_1,d_2)+L_0(|d_1-d_1'|+|d_2-d_2'|).
$$
Therefore,
$$
|\lambda_p(d_1,d_2)-\lambda_p(d_1',d_2')|\leq L_0(|d_1-d_1'|+|d_2-d_2'|).
$$
This proves the continuity of $\lambda_p(d_1,d_2)$.

Suppose that $(d_1, d_2) > (d_1', d_2')$ (i.e., $d_i > d_i', i=1,2$). By \eqref{d1-d2}, we have
\begin{align*}
\lambda_p(d_1, d_2) > &\ \frac{d'_1}{2e} \int_{L_1}^{L_2} \int_{L_1}^{L_2} J_1(x-y) (\phi_1(x) - \phi_1(y))^2 \rd x \rd y \\
&+ \frac{d'_2}{2G'(0)} \int_{L_1}^{L_2} \int_{L_1}^{L_2} J_2(x-y) (\phi_2(x) - \phi_2(y))^2 \rd x \rd y\\
&-\int_{L_1}^{L_2}\left[ \left(- \frac{a}{e} +\frac{d'_1}{e}\left( \int_{L_1}^{L_2}J_1(x-y)\rd y-1\right)\right) \phi_1^2(x)+2\phi_1(x)\phi_2(x)\right.\\
&\left.+\left(-\frac{b}{G'(0)}+\frac{d'_2}{G'(0)}\left(\int_{L_1}^{L_2}J_2(x-y)\rd y-1
\right)\phi_2^2(x) \right)\right]\rd x \\
\geq &\ \lambda_p(d_1', d_2'),
\end{align*}
where we have used the fact that
$$
\int_{L_1}^{L_2}\int_{L_1}^{L_2}J_i(x-y)\phi_i^2(x)\rd y\rd x<\int_{L_1}^{L_2}\phi_i^2(x)\rd x,~i=1,2.
$$
This proves the monotonicity.

(ii) Thanks to the properties of  $J_1$ and $J_2$, we have
\begin{align*}
\int_{L_1}^{L_2} \int_{L_1}^{L_2} J_i(x-y)\phi_i(x) \phi_i(y) \rd x \rd y &\leq \int_{L_1}^{L_2} \int_{L_1}^{L_2} J_i(x-y)\frac{\phi_i^2(x)+ \phi_i^2(y)}{2} \rd x \rd y \\
&\leq \int_{L_1}^{L_2} \phi_i^2(x)\rd x, ~ i=1,2.
\end{align*}
For any ${\boldsymbol\phi} = (\phi_1, \phi_2) \in \textit{\textbf{E}}$ and $\|{\boldsymbol\phi}\|_{\textit{\textbf{E}}}=1$, we obtain
\begin{align*}
\langle {\boldsymbol{\mathcal{K}}}{\boldsymbol\phi}, {\boldsymbol\phi} \rangle
=&\frac{d_1}{e} \int_{L_1}^{L_2} \int_{L_1}^{L_2} J_1(x-y) \phi_1(x) \phi_1(y)  \rd x \rd y + \frac{d_2}{G'(0)} \int_{L_1}^{L_2} \int_{L_1}^{L_2} J_2(x- y) \phi_2(x) \phi_2(y) \rd x \rd y\\
&+\int_{L_1}^{L_2}\left[ \left(-\frac{a}{e}- \frac{d_1}{e} \right)\phi_1^2(x) + 2 \phi_1(x) \phi_2(x) + \left( -\frac{b}{G'(0)} - \frac{d_2}{G'(0)} \right) \phi_2^2(x) \right]\rd x\\
\le& \int_{L_1}^{L_2} \left( -\frac{a}{e} \phi^2_1(x) + 2 \phi_1(x)\phi_2(x)- \frac{b}{G'(0)} \phi^2_2(x) \right) \rd x\\
=& \int_{L_1}^{L_2}{\boldsymbol\phi} \mathbf{A} {\boldsymbol\phi}^T dx \leq  \int_{L_1}^{L_2}\lambda_{\max} {\boldsymbol\phi} {\boldsymbol\phi}^Tdx
= \int_{L_1}^{L_2} \lambda_{\max}\left(\phi_1^2(x)+ \phi_2^2(x)\right) \rd x=\lambda_{\max},
\end{align*}
where we have used
\begin{equation*}
{\boldsymbol\phi} = (\phi_1, \phi_2), \ \mathbf{A} = \begin{pmatrix} -\frac{a}{e} & 1 \\ 1 & -\frac{b}{G'(0)} \end{pmatrix},\
 {\boldsymbol\phi} \mathbf{A} {\boldsymbol\phi}^T \leq \lambda_{\max} {\boldsymbol\phi} {\boldsymbol\phi}^T,
\end{equation*}
 and  the maximum eigenvalue of the matrix $\mathbf{A}$ is easily calculated to be
\[
\lambda_{\max}= \frac{-\frac{a}{e}-\frac{b}{G'(0)}+\sqrt{\left(\frac{a}{e}-\frac{b}{G'(0)}\right)^2+4}}{2}.
\]
Thus
\begin{align}\label{d1}
\lambda_{p}(d_1,d_2) = -\sup_{\|{\boldsymbol\phi}\|_{\textit{\textbf{E}}}=1} \langle {\boldsymbol{\mathcal{K}}}{\boldsymbol\phi}, {\boldsymbol\phi} \rangle \ge-\lambda_{\max}=\frac{\frac{a}{e}+\frac{b}{G'(0)}-\sqrt{\left(\frac{a}{e}-\frac{b}{G'(0)}\right)^2+4}}{2}.
\end{align}

On the other hand,
let
\[
\tilde{{\boldsymbol\phi}} = (\tilde{\phi_1}, \tilde{\phi_2}) = \left( \frac{m}{\sqrt{(L_2-L_1)(m^2 + n^2)}}, \frac{n}{\sqrt{(L_2-L_1)(m^2 + n^2)}} \right)
\]
with
{\small
\[
n=2,\ m = -\frac{a}{e}- \frac{d_1}{e}- \left(- \frac{b}{G'(0)}-\frac{d_2}{G'(0)} \right)+\sqrt{
\left[ \left(-\frac{a}{e}-\frac{d_1}{e}\right)-\left(-
\frac{b}{G'(0)}-\frac{d_2}{G'(0)}\right)        \right]^2+4}.
\]}
Then clearly,
$\|\tilde{{\boldsymbol\phi}}\|_{\textit{\textbf{E}}}=1$,
and
{\small \begin{align}\label{d2}
\lambda_{p}(d_1,d_2)
\le& -
\langle {\boldsymbol{\mathcal{K}}}\tilde{{\boldsymbol\phi}}, \tilde{{\boldsymbol\phi}} \rangle \nonumber\\
=&-\left[\frac{d_1}{e}\int_{L_1}^{L_2} \int_{L_1}^{L_2}J_1(x-y)\frac{m^2}{(L_2-L_1)(m^2+n^2)}\rd x\rd y\right.\nonumber\\
&\left.\ \ \ \ \ \ +\frac{d_2}{G'(0)}\int_{L_1}^{L_2} \int_{L_1}^{L_2}J_2(x-y)\frac{n^2}{(L_2-L_1)(m^2+n^2)}\rd x\rd y    \right]\nonumber\\
&-\left[\int_{L_1}^{L_2}\left(\left(-\frac{a}{e}-\frac{d_1}{e}\right)\frac{m^2}{(L_2-L_1)(m^2+n^2)}+2\frac{mn}{(L_2-L_1)(m^2+n^2)}   \right.\right.\nonumber\\
&\left.\left.\ \ \ \ \  \ \ \ \ \ \ +\left(-\frac{b}{G'(0)}-\frac{d_2}{G'(0)}\right)\frac{n^2}{(L_2-L_1)(m^2+n^2)}\right)   \rd x \right]\nonumber\\
\le& -\frac{\left(-\frac{a}{e}-\frac{d_1}{e}\right)m^2+2mn+\left(-\frac{b}{G'(0)}-\frac{d_2}{G'(0)}\right)n^2}{m^2+n^2}\nonumber\\
= &-\frac{\left(-\frac{a}{e}-\frac{d_1}{e}\right)+\left(-\frac{b}{G'(0)}-\frac{d_2}{G'(0)}\right) -\sqrt{\left[\left(-\frac{a}{e}-\frac{d_1}{e}\right)-\left(-\frac{b}{G'(0)}-\frac{d_2}{G'(0)}\right)         \right]^2 +4}}{2}\nonumber\\
=&\frac{\frac{a}{e}+\frac{b}{G'(0)}+\frac{d_1}{e}+\frac{d_2}{G'(0)}-\sqrt{ \left( \frac{a}{e}- \frac{b}{G'(0)}+\frac{d_1}{e}-\frac{d_2}{G'(0)}\right)^2      +4}}{2}.
\end{align}}
Letting \(d_1, d_2 \to 0\) in \(\eqref{d2}\) we deduce
\[
\limsup_{(d_1, d_2) \to (0,0)} \lambda_{p(d_1, d_2)} \leq \frac{\frac{a}{e} + \frac{b}{G'(0)} - \sqrt{\left( \frac{a}{e} - \frac{b}{G'(0)} \right)^2 + 4}}{2}.
\]
Combining this with \(\eqref{d1}\), we arrive at the desired conclusion (ii).

(iii) The detailed proof of this result can be found in \cite[Proof of Theorem1.3]{fb8}.\qed
\subsection{A fixed boundary problem}
We consider the corresponding fixed boundary problem of \eqref{A}:
\begin{equation}\label{m0}
\begin{cases}
u_{t}=d_{1}\displaystyle\int_{L_1}^{L_2}J_{1}(x-y)u(t,y)\rd y-d_{1}u-au+ev,&t>0,x \in (L_1,L_2),\\
v_{t}=d_{2}\displaystyle\int_{L_1}^{L_2}J_{2}(x-y)v(t,y)\rd y-d_{2}v-bv+G(u),&t>0,x \in (L_1,L_2),\\
u(0,x)=u_{0}(x),v(0,x)=v_{0}(x),&x \in [L_1,L_2],
\end{cases}
\end{equation}
where $-\infty<L_1<L_2<+\infty$, and initial data $u_0,v_0\in C([L_1,L_2])$ are nonnegative and not identically $0$ simultaneously. It is well known that problem \eqref{m0} admits a unique global positive solution (see, e.g., \cite{fb1}). It's long-time dynamical behavior is given in the following result.
\begin{proposition}\label{990}
Suppose that {\bf{(J)}} holds and $G(u)$ satisfies {\bf(G1)-(G2)}. Let $(u,v)$ be the unique positive solution of \eqref{m0}, $\lambda_p(L_1,L_2)$ be the principal eigenvalue of \eqref{c4}, and $(u^*,v^*)$ be given in \eqref{q} when $\mathbb R_0>1$. Then the following statements are valid:
\begin{itemize}
    \item[(i)] If $\lambda_p(L_1,L_2) < 0$, then problem \eqref{m0} admits a unique positive steady state solution $(\tilde{u},\tilde{v})=(\tilde{u}_{(L_1,L_2)},\tilde{v}_{(L_1,L_2)})\in C([L_1,L_2])\times C([L_1,L_2])$, moreover,
    \[
    \begin{cases}
    \mbox{$(u(t,x),v(t,x))\to(\tilde{u},\tilde{v})$ as $t\to \infty$ uniformly for $x\in[L_1,L_2],$}\\
    \mbox{$0< \tilde{u} (x) \leq u^*$, $0< \tilde{v}(x) \leq v^* $ in $[L_1,L_2]$.}
    \end{cases}
    \]
   \item[(ii)]  If $\lambda_p(L_1,L_2)\geq0$, then $(0,0)$ is the only nonnegative steady state of \eqref{c4}, and
   \[
   \mbox{$(u(t,x),v(t,x))\to(0,0)$ as $t\to\infty$ uniformly for $x\in[L_1,L_2].$}
   \]
   \item[(iii)] If $\mathcal{R}_0 > 1$, and hence $\lambda_p(L_1,L_2) < 0$ for $L_2-L_1\gg1$, then
   \[
   \lim\limits_{-L_1,L_2\to+\infty}(\tilde{u}_{(L_1,L_2)}(x),\tilde{v}_
       {(L_1,L_2)}(x))=(u^{*},v^{*})~ \text{locally uniformly for } ~ x\in\mathbb{R}.\]
\end{itemize}
\end{proposition}
\noindent{\bf{Proof}}. Assume that $\lambda_p(L_1,L_2)<0$. We first show that $(\delta \phi_1, \delta\phi_2)$ is a lower solution of the steady state problem of \eqref{m0} for sufficiently small $\delta > 0$, where $(\phi_1, \phi_2)$ is a positive eigenfunction pair corresponding to $\lambda_p(L_1,L_2)$. Thanks to $\lim\limits_{\delta \to 0}(\frac{G(\delta\phi_1)}{\delta\phi_1}-G'(0))=0$, {\bf{(G2)}} and \eqref{c4}, we can choose $\delta >0$ small enough such that
\begin{equation*}
d_1\int_{L_1}^{L_2}J_1(x-y)\delta \phi_1(y)\rd y-d_1\delta\phi_1-a\delta\phi_1+e\delta\phi_2=-\lambda_p(L_1,L_2)e\delta\phi_1>0,
\end{equation*}
and
\begin{align*}
&d_2\int_{L_1}^{L_2}J_2(x-y)\delta \phi_2(y)\rd y-d_2\delta\phi_2-b\delta\phi_2+G(\delta\phi_1)\\
&=-\lambda_p(L_1,L_2)G'(0)\delta\phi_2+\delta\phi_1(\frac{G(\delta\phi_1)}
{\delta\phi_1}-G'(0))\\
&\geq\delta\left[-\lambda_p(L_1,L_2)G'(0)\min_{x \in [L_1,L_2]}\phi_2(x)+(\frac{G(\delta\phi_1)}{\delta\phi_1}-G'(0))\max_{x \in [L_1,L_2]}\phi_1(x)\right]\geq 0.
\end{align*}
This implies that $(\delta\phi_1, \delta\phi_2)$ is a lower solution. Next, we can argue as in the proof of \cite[Proposition 3.7]{fb4} to obtain the desired conclusions in (i). Moreover, following the proof of \cite[Proposition 3.5]{fb1}, we can obtain (ii) and (iii). Since only obvious modifications are needed, we omit the details.\qed

\section{Spreading-vanishing dichotomy and criteria}
Recall that we always assume {\bf{(J)}} and {\bf{(W)}} hold, {\bf{\(G(u)\)}} satisfies {\bf(G1)-(G2)}, the initial function pair \((u_0, v_0)\) satisfies \(\eqref{B}\). Let  \((u, v, g, h)\) be the unique positive solution of \(\eqref{A}\). According to Theorem \ref{B2}, we can define
\[
g_{\infty} := \lim_{t \to \infty} g(t)\in [-\infty, -h_0), \quad h_{\infty} := \lim_{t \to \infty} h(t)\in (h_0, \infty].
\]
The conclusions in Theorems \ref{T12}, \ref{T13}, \ref{T14}  and \ref{thm6} will follow from the results to be proved below.
\subsection{Spreading-vanishing dichotomy}
\begin{lemma}\label{d3}
If \( h_{\infty} - g_{\infty} < \infty \), then
$$
\lim_{t \to \infty} \|u\|_{C([g(t), h(t)])}= \lim_{t \to \infty} \|v\|_{C([g(t), h(t)])}= 0
$$
and
$$\lambda_p(g_{\infty}, h_{\infty}) \geq 0,$$
where \(\lambda_p(g_{\infty}, h_{\infty})\) is the principal eigenvalue of \eqref{c4} with \([L_1, L_2]=[g_{\infty}, h_{\infty}]\).
\end{lemma}
\noindent{\bf{Proof}}.
We first prove that \(\lambda_p(g_{\infty}, h_{\infty}) \geq 0\). Suppose, to the contrary, that \(\lambda_p(g_{\infty}, h_{\infty}) < 0\). Then, by  Proposition \ref{c3} and assumptions  {\bf{(J)}} and {\bf{(W)}}, there exists a sufficiently large  \(T_{0} > 0\) such that
\[
\lambda_p(g(T), h(T)) < 0, \quad |g(T) - g_{\infty}| < \varepsilon_1  \quad\text{and}\quad |h(T) - h_{\infty}| < \varepsilon_1, \quad T>T_{0},
\]
where \(\varepsilon_1 \in (0, h_0/4)\) is small enough such that
\[
J_1(x) > 0 \quad\mbox{for } x \in [-4\varepsilon_1, 4\varepsilon_1] \quad\mbox{and}\quad\int_{\varepsilon_1}^{h_0} W(x) \rd x > 0.
\]

According to Proposition \ref{990}, the solution \((u_1(t, x), v_1(t, x))\) of the  following problem
\begin{equation}\label{d4}
\begin{cases}
(u_1)_t = d_1 \displaystyle\int_{g(T)}^{h(T)} J_1(x-y) u_1(t, y) \rd y - d_1 u_1 - a u_1 + e v_1, & t > T,\  x \in [g(T), h(T)], \\
(v_1)_t = d_2 \displaystyle\int_{g(T)}^{h(T)} J_2(x-y) v_1(t, y) \rd y - d_2 v_1 - b v_1 + G(u_1), & t > T,\  x \in [g(T), h(T)], \\
u_1(T, x) = u(T, x), \quad v_1(T, x) = v(T, x), & x \in [g(T), h(T)],
\end{cases}
\end{equation}
converges to the unique positive steady state \((\tilde{u}_1, \tilde{v}_1)\) of \eqref{d4} uniformly in \([g(T), h(T)]\) as \(t \to +\infty\). Moreover, a simple comparison argument gives
$$u(t,x)\ge u_1(t,x), \quad v(t,x)\ge v_1(t,x)\quad\text{for} ~ t>T ~ \text{and}~ x\in [g(T), h(T)].$$
Thus, there exists \( T_1 > T \) such that
\[
u(t,x) \geq \frac{1}{2} \tilde{u}_1(x) > 0, \quad v(t,x) \geq \frac{1}{2} \tilde{v}_1(x) > 0 \quad \text{for } t > T_1 \text{ and } x \in [g(T), h(T)].
\]
Now, we denote
\[
\delta_1 := \inf_{x \in [-4\varepsilon_1, 4\varepsilon_1]} J_1(x) > 0, \quad \delta_2 := \inf_{x \in [g(T), h(T)]} \tilde{u}_1(x) > 0,
\quad \delta_3 := \inf_{x \in [g(T), h(T)]} \tilde{v}_1(x) > 0.
\]
Then for \( t > T_1 \), we have
\begin{align*}
h'(t)&= \mu \int_{g(t)}^{h(t)}\int_{h(t)}^{+\infty}J_1(x-y) u(t,x) \rd y \rd x + \mu\rho \int_{g(t)}^{h(t)} v(t,x) W(h(t)-x) \rd x \\
&\geq \mu \int_{g(T)}^{h(T)}\int_{h(t)}^{h(t)+2\varepsilon_1} J_1(x-y) u(t,x) \rd y \rd x + \mu \rho\int_{g(T)}^{h(T)} v(t,x) W(h(t)-x) \rd x\\
&\geq \frac{1}{2}\mu \int_{g(T)}^{h(T)}\int_{h(t)}^{h(t)+2\varepsilon_1} J_1(x-y) \tilde{u}_1(x) \rd y \rd x + \frac{1}{2}\mu \rho\int_{g(T)}^{h(T)} \tilde{v}_1(x) W(h(t)-x) \rd x\\
&\geq \frac{1}{2}\mu\delta_2 \int_{h(t)-2\varepsilon_1}^{h(t)-\varepsilon_1}\int_{h(t)}^{h(t)+2\varepsilon_1} J_1(x-y) \rd y \rd x + \frac{1}{2}\mu \rho\delta_3\int_{h(t)-h(T)}^{h(t)-g(T)} W(y) \rd y\\
&\geq \mu\varepsilon^2\delta_1 \delta_2 +\frac{1}{2}\mu\rho\delta_3\int_{\varepsilon_1}^{h_0}W(y)\rd y>0,
\end{align*}
where we have used  the fact $\left[ h(t)-2\varepsilon_1, h(t)-\varepsilon_1  \right]\subset [g(T), h(T)]$ for $t\ge T$.
This implies \( h_{\infty} = +\infty \),  which contradicts the assumption that  \( h_{\infty} - g_{\infty} < +{\infty} \). Therefore, \( \lambda_p(g_{\infty}, h_{\infty}) \geq 0 \).

Next, we show that \( (u(t,x), v(t,x)) \to (0, 0) \) uniformly in \( [g(t), h(t)] \) as \( t \to \infty \).  Denote by \( (u_2(t,x), v_2(t,x)) \)the unique solution of
\[
\begin{cases}
(u_2)_t = d_1 \displaystyle\int_{g_\infty}^{h_\infty} J_1(x-y) u_2(t,y) \rd y - d_1 u_2 - au_2 + e v_2, & t > 0, \, x \in [g_\infty, h_\infty], \\
(v_2)_t = d_2 \displaystyle\int_{g_\infty}^{h_\infty} J_2(x-y) v_2(t,y) \rd y - d_2 v_2 - b v_2 + G(u_2), & t > 0, \, x \in [g_\infty, h_\infty], \\
u_2(0,x) = u_{20}(x), \quad v_2(0,x) = v_{20}(x), & x \in [g_\infty, h_\infty],
\end{cases}
\]
where \( u_{20}(x) = u_0(x) \), \( v_{20}(x) = v_0(x) \) if $x\in[-h_0,h_0]$ and $u_{20}(x) =v_{20}(x)=0$ if $|x|> h_0$. By Lemma \ref{vv}, we have
\[
0\leq u(t,x) \leq u_2(t,x), \quad 0\leq v(t,x) \leq v_2(t,x) \quad \text{for } t > 0 \text{ and } x \in [g(t), h(t)].
\]
Since $\lambda_{p}(g_\infty, h_\infty) \geq 0$, Proposition \ref{990} infers that $(u_2(t, x), v_2(t, x)) \rightarrow (0,0)$ uniformly for $x \in [g_\infty, h_\infty]$ as $t \rightarrow \infty$, and hence
\[
\lim_{t \rightarrow \infty} (u(t, x), v(t, x)) = (0,0) \quad \text{uniformly for } x \in [g(t), h(t)].
\]
This completes the proof. \qed

\begin{lemma} \label{d5}
If $\mathcal{R}_0 \le 1$, then
\[
\lim_{t \rightarrow \infty} \|u\|_{C([g(t), h(t)])} = \lim_{t \rightarrow \infty} \|v\|_{C([g(t), h(t)])} = 0.
\]
\end{lemma}
\noindent{\bf{Proof.}} Let $(\bar{u}(t), \bar{v}(t))$ be the solution of the following ODE problem
\begin{align}\label{d6}
\left\{
\begin{aligned}
&u'(t) = -au + ev, & t > 0, \\
&v'(t) = -bv + G(u), & t > 0, \\
&u(0) = \|u_0\|_\infty, v(0) = \|v_0\|_\infty.
\end{aligned}
\right.
\end{align}
If $\mathcal{R}_0 = \frac{eG'(0)}{ab} < 1$, it then follows from \cite{fb13} that the solution $(0,0)$ of problem \eqref{d6} is globally attractive and thus
\[
(\bar{u}(t), \bar{v}(t)) \rightarrow (0,0) \quad \text{as } t \rightarrow \infty.
\]
If $\mathcal{R}_0= 1$, then we define
\[
V(t)=\alpha \bar u(t)+\bar v(t) \quad \mbox{ for }t>0,
\]
where $\alpha>0$ is a constant to be determined.
Therefore, $V(t)\geq 0$ for $t\geq 0$.
Moreover,
\begin{align*}
   V'(t)=&\ \alpha \bar u'(t)+\bar v'(t) \\
   =&\ \alpha \lf[-a \bar u+e\bar v\rr]-b\bar v+G(\bar u)\\
   =&-\alpha a \bar u+\lf[\alpha e \bar v-b \bar v\rr]+G(\bar u).
\end{align*}
Let $\alpha:=\frac{G'(0)}{a}>0$; then
\[
V'(t)=-G'(0)\bar u+G(\bar u)\leq 0 \quad\mbox{for }\bar u\geq 0
\]
due to $\mathcal{R}_0=1$ and {\bf(G2)}. Therefore, $V_{\infty}:=\lim_{t\rightarrow \infty}V(t)$ exists.
Now we claim $\lim_{t\rightarrow \infty}\bar u(t)=0$. Otherwise, there exist a constant $\delta>0$ and a sequence $t_n>0$ such that
$\bar u(t_n)\geq \delta>0$ for all large $n$.
Due to $G({u})<G'(0){u}$ for $u>0$,
\[
V'(t_n)=-G'(0)\bar u(t_n)+G(\bar u(t_n))\leq-G'(0)\delta+G(\delta)<0
\]
for all large $n$, which contradicts the fact that $V_{\infty}$ exists. Hence, we have $\lim_{t\rightarrow \infty}\bar u(t)=0$. Analogously we can show that $\lim_{t\rightarrow \infty}\bar v(t)=0$.

Moreover, by the comparison principle (Lemma \ref{vv}), we obtain
\begin{equation}\label{777}
u(t, x) \leq \bar{u}(t), \quad v(t, x) \leq \bar{v}(t) \quad \text{for } t \geq 0 \text{ and } x \in [g(t), h(t)].
\end{equation}
Therefore, \[
\lim_{t \rightarrow \infty} \max_{x \in [g(t), h(t)]} u(t, x) = \lim_{t \rightarrow \infty} \max_{x \in [g(t), h(t)]} v(t, x) = 0.
\]
This completes the proof. \qed

\begin{lemma}\label{d7}
If $\mathcal{R}_0 > 1$, then $h_\infty < +\infty$ if and only if $-g_\infty < +\infty$.
\end{lemma}
\noindent{\bf{Proof.}}  We argue by contradiction. Suppose, without loss of generality, that \( h_\infty = +\infty \) and \( -g_\infty < +\infty \). Using $\mathcal{R}_0 > 1$, Proposition \ref{c3} and $h_\infty = +\infty$,  we can find a large constant $T > 0$ such that
$\lambda_{p}(g(t), h(t)) < 0$ for  $t\geq T$. Hence we can repeat the  argument in the proof of Lemma \ref{d3} to obtain constants $\delta > 0$ and $T_1 > T$ such that
\[
g'(t) = -\mu \displaystyle\int_{g(t)}^{h(t)} \int_{-\infty}^{g(t)} J_1(x-y) u(t, x) \rd y \rd x - \mu\rho \int_{g(t)}^{h(t)} v(t, x) W(x-g(t))\rd x < -\delta \quad \text{for } t > T_1.
\]
This is a contradiction to $-g_\infty < +\infty$. \qed
\begin{lemma}\label{d8}
 If $\mathcal{R}_0 > 1$ and $h_\infty - g_\infty = +\infty$, then
\[
\lim_{t \rightarrow \infty} (u(t, x), v(t, x)) = (u^*, v^*) \quad \text{locally uniformly in } \mathbb{R},
\]
where $(u^*, v^*)$ is defined in \eqref{q}.
\end{lemma}
\noindent{\bf{Proof.}} The proof is similar to that in previous works (see, e.g., \cite[Lemma 4.6]{fb4}, so we omit it here. \qed

\vspace{1em}

 Clearly Theorem \ref{T12} is an immediate consequence of Lemmas \ref{d3}, \ref{d5} and \ref{d8}.
\subsection{Spreading-vanishing criteria}
\begin{lemma}\label{666}
If $\mathcal{R}_0 <1$, then $h_\infty-g_\infty<\infty$.
\end{lemma}

\noindent{\bf{Proof}}. Since \( \mathcal{R}_0 < 1 \), it is well known that the solution \((\bar{u}(t), \bar{v}(t))\) of the ODE system \eqref{d6} decays exponentially. Therefore, there exist constants \( c_1 > 0 \) and \( \delta > 0 \) independent of \( t \) such that
$$
\bar{u}(t)\leq c_1e^{-\delta t},\quad\bar{v}(t)\leq c_1e^{-\delta t}\quad \textnormal{for}~t\geq0.
$$
It then follows from \eqref{777} that
$$
u(t, x) \leq c_1e^{-\delta t}, \quad v(t, x) \leq c_1e^{-\delta t} \quad \text{for } t \geq 0 \text{ and } x \in [g(t), h(t)].
$$
Therefore, by \eqref{A}, we have
\begin{align*}
h'(t)&= \mu \int_{g(t)}^{h(t)}\int_{h(t)}^{+\infty}J_1(x-y) u(t,x) \rd y \rd x + \mu\rho \int_{g(t)}^{h(t)} v(t,x) W(h(t)-x) \rd x \\
&\leq\mu c_1e^{-\delta t}(h(t)-g(t))+\mu\rho \|W\|_\infty c_1e^{-\delta t}(h(t)-g(t))\\
&\leq c_2 e^{-\delta t}(h(t)-g(t)),
\end{align*}
where the constant $c_2>0$ is independent of $t$. Analogously we have
$$
-g'(t)\leq c_2 e^{-\delta t}(h(t)-g(t)).
$$
Thus, we immediately obtain
$$
(h(t)-g(t))'\leq 2c_2 e^{-\delta t}(h(t)-g(t))\quad \text{for}~t\geq 0,
$$
which leads to
$$
h(t)-g(t)\leq c_3 e^{-\frac{c_4}{\delta}e^{-\delta t}}\quad \text{for}~t\geq 0
$$
with constants $c_3,c_4>0$ independent of $t$. Letting $t \to \infty$, we have
$$
h_\infty-g_\infty\leq c_3<\infty.
$$
The  proof is complete.\qed

\begin{lemma}\label{R0-1}
If $\mathcal{R}_0 =1$, then $h_\infty-g_\infty<\infty$ provided one of the following holds:
\begin{itemize}
\item[(i)] $W(x)\leq \eta W_{J_2}(x)$ for all $x\geq 0$ and some constant $\eta>0$,  or
\item[(ii)] $g(u)$ in \eqref{R0=1} satisfies $g(u)\geq \sigma u^\lambda$ for some $\sigma>0,\ \lambda\in (0,1)$ and all small $u>0$.
\end{itemize}
\end{lemma}
\noindent{\bf{Proof}}.
If (i) is satisfied, then the conclusion follows directly from a comparison argument and Lemma 4.2 of \cite{fb31}.

Suppose now (ii) holds and let \((\bar{u}(t), \bar{v}(t))\) be the solution of  \eqref{d6}. From the proof of Lemma \ref{d5} we know that
\[
(\bar{u}(t), \bar{v}(t)) \rightarrow (0,0) \quad \text{as } t \rightarrow \infty.
\]
We claim that (ii) implies
\begin{equation}\label{int-finite}
\int_0^\infty \bar u(t)dt<\infty,\ \int_0^\infty \bar v(t)dt<\infty.
\end{equation}
Before giving the proof of \eqref{int-finite}, let us see how it leads to the desired conclusion of the lemma.
It  follows from \eqref{777} and \eqref{A} that
\begin{align*}
h'(t)&= \mu \int_{g(t)}^{h(t)}\int_{h(t)}^{+\infty}J_1(x-y) u(t,x) \rd y \rd x + \mu\rho \int_{g(t)}^{h(t)} v(t,x) W(h(t)-x) \rd x \\
&\leq\mu \bar u(t)(h(t)-g(t))+\mu\rho \|W\|_\infty \bar v(t)(h(t)-g(t))\\
&\leq c_2 [\bar u(t)+\bar v(t)](h(t)-g(t)),
\end{align*}
where the constant $c_2>0$ is independent of $t$. Analogously we have
$$
-g'(t)\leq c_2 [\bar u(t)+\bar v(t)](h(t)-g(t)).
$$
Hence
$$
(h(t)-g(t))'\leq 2c_2[\bar u(t)+\bar v(t)](h(t)-g(t))\quad \text{for}~t\geq 0.
$$
It follows that
\[
\ln\Big(\frac{h(t)-g(t)}{2h_0}\Big)\leq 2c_2\int_0^t[\bar u(s)+\bar v(s)]ds<2c_2\int_0^\infty[\bar u(s)+\bar v(s)]ds<\infty \mbox{ for all } t>0,
\]
which clearly implies the desired conclusion.

It remains to prove \eqref{int-finite}. The assumption $\mathcal R_0=1$ and (ii) allow us to write
\[
\begin{cases} \bar u'=-a\bar u+e\bar v,\\
\bar v'=-b\bar v+\frac{ab}e \bar v-g(\bar u)\bar u,
\end{cases} \ \ \  t>0.
\]
Let $w:=b\bar u+e\bar v$; we easily obtain
\[
\bar u'=-(a+b)\bar u+w,\ w'=-eg(\bar u)\bar u,
\]
and so
\[
\bar u''+(a+b)\bar u'+eg(\bar u)\bar u=0.
\]
Hence $(U,V)=(\bar u, \bar u')=(U, U')$ satisfies
\[
U'=V,\ V'=-(a+b)V-eg(U)U.
\]
Moreover, we know $(U(t), V(t))\to (0, 0)$ as $t\to\infty$, and $U(t)>0$ for $t>0$.
We now investigate how the solution curve $t\to (U(t), V(t))$ approaches $(0,0)$ in the $UV$-plane via a phase plane analysis.

Note that $U(t)>0$ and $U(t)\to 0$ as $t\to\infty$ imply the existence of a positive sequence $t_n\to\infty$ such that
$V(t_n)=U'(t_n)<0$ for all $n\geq 1$. For any $\beta>0$ and $n\geq 1$ we consider the curve $\ell^\beta_n$ given by
\[
\ell^\beta_n:=\{(U,V)\in\mathbb R^2: V=-\beta g(U)U,\ \  U\in [0,\, U(t_n)]\}.
\]
Our assumption on $g(u)$ implies that $\ell_n^\beta$ is a monotone curve lying below the $U$-axis connecting the two points  $(0,0)$ and $P_n:=(U(t_n), -\beta g(U(t_n))U(t_n))$.

Define
\[\begin{cases}
\ell_n:=\{(U,V)\in \mathbb R^2: U=U(t), V=V(t),\ t\geq t_n\},\\
 Q_n:=(U(t_n), V(t_n)),\
\beta_n:=\min\Big\{\frac{-V(t_n)}{2g(U(t_n))U(t_n)}, \frac 1n\Big\}.\\
\end{cases}
\]
Clearly $\ell_n$ is a piece of the solution curve $t\to (U(t), V(t))$ that connects $(0,0)$ and $Q_n$. It is easily seen from our choice of $\beta_n$  that $Q_n$ is below $P_n$.
For $t>t_n$ but close to $t_n$, $U'(t)=V(t)<0$ and so the point $(U(t), V(t))\in \ell_n$ moves leftward from $Q_n$ but stays below $\tilde \ell_n$ as $t$ is increased but close to $t_n$. We show below that for sufficiently large $n$,  $(U(t), V(t))$ stays below $\tilde\ell_n:=\ell_n^{\beta_n}$ in the $UV$-plane for all $t>t_n$.  Indeed, as the point $(U(t), V(t))$ moves leftward along $\ell_n$ with increasing $t>t_n$,  either it  stays below $\tilde \ell_n$ all the way
till it reaches $(0,0)$
as $t\to\infty$, or there is a first time moment $t_n^*>t_n$ such that $\ell_n$ intersects $\tilde \ell_n$ at $P_n^*:=(U(t_n^*), V(t_n^*))$, where $\ell_n$ has slop
\[
{\rm slope}[\ell_n](P_n^*)=\frac{V'(t_n^*)}{U'(t_n^*)}=-(a+b)-\frac{eg(U(t_n^*))U(t_n^*)}{V(t_n^*)}=-(a+b)+\frac{e}\beta_n.
\]
On the other hand, due to $0<U(t_n^*)<U(t_n)$ and $g'(u)u=\frac{ab}e-g(u)-G'(u)\to \frac{ab}e-G'(0)=0$ as $u\to 0$,  we have
\[
{\rm slope}[\tilde \ell_n](P_n^*)=-\beta_n[g'(U(t_n^*))U(t_n^*)+g(U(t_n^*))]=o(1) \mbox{ as } n\to\infty.
\]
 The fact that $ \ell_n$ intersects $\tilde \ell_n$ at $P_n^*$ from below when moved leftward along $\ell_n$ implies that
\[
{\rm slope}[\tilde \ell_n](P_n^*)\geq {\rm slope}[\ell_n](P_n^*).
\]
It follows that
\[
 o(1)\geq -(a+b)+\frac{e}{\beta_n} \mbox{ as } n\to\infty.
\]
But this is impossible since $\frac e{\beta_n}\to +\infty$ as $n\to\infty$.  This contradiction shows that $ \ell_n$ lies below $\tilde \ell_n$ for all large $n$. Fix such a large $n$ so that (ii) holds for $u\in (0, U(t_n)]$.  Then
\[
V(t)< -\beta_n g(U(t))U(t)  \mbox{ for } t\geq  t_n,
\]
and by (ii) we obtain,
\[
U'=V< -\beta_n \sigma U^{1+\lambda} \mbox{ for } t\geq  t_n.
\]
It follows that
\[
(U^{-\lambda})'> (1+\lambda)\beta_n\sigma \mbox{ for } t\geq  t_n,
\]
which yields
\[
U(t)< \Big[U^{-\lambda}( t_n)+(1+\lambda)\beta_n\sigma (t- t_n)\Big]^{-1/\lambda} \mbox{ for } t> t_n.
\]
Therefore $\bar u(t)=U(t)$ is integrable over $[0,\infty)$.
From $\bar u'=-a\bar u+e\bar v$ we obtain
\[
e\int_0^t\bar v(s)ds=\bar u(t)-\bar u(0)+a\int_0^t\bar u(s)ds\to -\bar u(0)+a\int_0^\infty \bar u(s)ds \mbox{ as } t\to\infty.
\]
So $\bar v(t)$ is also integrable and
\eqref{int-finite} is  proved.
\qed

\begin{lemma}\label{d9}
If $\mathcal{R}_0 \geq \left(1 + \frac{d_1}{a}\right)\left(1 + \frac{d_2}{b}\right)$, then spreading always happens for \eqref{A}.
\end{lemma}
\noindent{\bf{Proof}}. By Proposition \ref{c3}, we see $\lambda_{p}(L_1, L_2) < 0$ for any finite interval $[L_1, L_2]$. The desired conclusion then follows from Lemma \ref{d3}. \qed

\begin{lemma}\label{e1}
Suppose that $1 < \mathcal{R}_0 < \left(1 + \frac{d_1}{a}\right)\left(1 + \frac{d_2}{b}\right)$ and $h_0 \geq L^*$ hold. Then spreading always happens for \eqref{A}.
\end{lemma}
\noindent{\bf{Proof}}. Otherwise, we have $h_\infty - g_\infty < +\infty$. Since $h_\infty - g_\infty \geq 2h_0 \geq 2L^*$,  it follows from Proposition~\ref{c3} that $\lambda_{p}(g_\infty, h_\infty) < 0$. However,  Lemma \ref{d3}  shows $\lambda_{p}(g_\infty, h_\infty) \geq 0$, leading to a contradiction. Hence spreading always occurs for problem~\eqref{A} when \( h_0 \geq L^* \).  \qed

\begin{lemma}\label{e2}
Suppose that $1<\mathcal{R}_0<\left(1+\frac{d_1}{a}\right)\left(1+\frac{d_2}{b}\right)$ and $h_0<L^*$ hold. Then there exists $\underline{\mu}>0$ depending on $(u_0, v_0)$ such that  vanishing happens for \eqref{A} if $0<\mu \leq \underline{\mu}$.
\end{lemma}
\noindent{\bf Proof.} We prove this lemma by constructing an appropriate upper solution. Let $\lambda_p(h_0)$ be the principal eigenvalue of \eqref{c4} with $[L_1, L_2]$ replaced by $[-h_0, h_0]$. Since $h_0<L^*$, we have $\lambda_p(h_0) > 0$. By Proposition \ref{c3}, $\lambda_p(-L,L) $ is continuous in $L$ and hence there exists $h_1>h_0$ but close to $h_0$ such that $\lambda_p(h_1) > 0$. Let $(\phi_1, \phi_2)$ be a positive  eigenfunction pair corresponding to $\lambda_p(h_1)$,
\[
\Lambda :=\frac{\lambda_p(h_1)}{2} \min\{e,G'(0)\}, ~ C := h_1 - h_0,\]
\[ M:= C \Lambda  \left[ \int_{-h_1}^{h_1} \phi_1(x) \rd x + \rho \max_{x\in[0,2h_1]} W(x) \int_{-h_1}^{h_1} \phi_2(x) \rd x \right]^{-1},
\]
and for  $t \geq 0$ and $x \in [-h_1, h_1]$, define
\[\begin{cases}
\bar{h}(t) := h_0 + C(1 - e^{-\Lambda t}), \quad \bar{g}(t): = -\bar{h}(t),
\\
\bar{u}(t, x) := \mu^{-1} Me^{-\Lambda t} \phi_1(x), \quad \bar{v}(t, x) := \mu^{-1}M e^{-\Lambda t} \phi_2(x).
\end{cases}
\]
 We next show that $(\bar{u}, \bar{v}, \bar{g}, \bar{h})$ is an upper solution of \eqref{A} when $\mu>0$ is small enough. Obviously,
\[
-\bar{g}(0) = \bar{h}(0) = h_0, \quad [\bar{g}(t), \bar{h}(t)] \subset [-h_1, h_1] \quad \text{for } t \geq 0,
\]
and
\[
\bar{u}(t, x) > 0, \quad \bar{v}(t, x) > 0 \quad \text{for } x = \bar{g}(t) \text{ or } \bar{h}(t), \quad t > 0.
\]
By direct computations, we have
\begin{align*}
&\bar{u}_t - d_1 \int_{\bar{g}(t)}^{\bar{h}(t)} J_1(x-y) \bar{u}(t, y) \rd y + d_1 \bar{u}(t, x) + a \bar{u}(t, x) - e \bar{v}(t, x) \\
&\ge\bar{u}_t - d_1 \int_{-h_1}^{h_1} J_1(x-y) \bar{u}(t, y) \rd y + d_1 \bar{u}(t, x) + a \bar{u}(t, x) - e \bar{v}(t, x) \\
&= \mu^{-1}M e^{-\Lambda t} \phi_1(x) \left[ \lambda_p(h_1)e - \Lambda \right] \geq 0 \quad \text{for } t > 0 \text{ and } x \in (\bar{g}(t), \bar{h}(t)).
\end{align*}
Similarly, we have
\begin{align*}
&\bar{v}_t - d_2 \int_{\bar{g}(t)}^{\bar{h}(t)} J_2(x-y) \bar{v}(t, y) \rd y + d_2 \bar{v}(t, x) + b \bar{v}(t, x) -G(\bar{u}(t, x)) \\
&\geq \mu^{-1}M e^{-\Lambda t} \phi_2(x) \left[ \lambda_p(h_1)G'(0) - \Lambda \right] \geq 0 \quad \text{for } t > 0 \text{ and } x \in (\bar{g}(t), \bar{h}(t)),
\end{align*}
where we have used $G(\bar{u}) \leq G'(0) \bar{u}$ due to \textbf{(G2)}.

Moreover, we have, by the definition of $M$,
\begin{align*}
&\mu \int_{\bar{g}(t)}^{\bar{h}(t)} \int_{\bar{h}(t)}^{+\infty} J_1(x-y) \bar{u}(t,x) \rd y  \rd x + \mu \rho \int_{\bar{g}(t)}^{\bar{h}(t)} \bar{v}(t,x) W(\bar{h}(t)-x)  \rd x \\
&\leq \mu \int_{-h_1}^{h_1} \bar{u}(t,x)  \rd x + \mu \rho \max_{x \in [0,2h_1]} W(x) \int_{\bar{g}(t)}^{\bar{h}(t)} \bar{v}(t,x)  \rd x \\
&\leq M e^{-\Lambda t} \left( \int_{-h_1}^{h_1} \phi_1(x)  \rd x + \rho \max_{x \in [0,2h_1]} W(x) \int_{-h_1}^{h_1} \phi_2(x)  \rd x \right) \\
&= C\Lambda e^{-\Lambda t} = \bar{h}'(t) \quad \text{for } t > 0.
\end{align*}
Analogously,
\begin{align*}
&\bar{g}'(t) \leq -\mu \int_{\bar{g}(t)}^{\bar{h}(t)} \int_{-\infty}^{\bar{g}(t)} J_1(x-y) \bar{u}(t,x)  \rd y  \rd x - \mu \rho \int_{\bar{g}(t)}^{\bar{h}(t)} \bar{v}(t,x) W(x-\bar{g}(t))  \rd x \quad \text{for } t > 0.
\end{align*}
Moreover, for any given admissible initial datum $(u_0, v_0)$,
\begin{align}\label{e3}
\|u_0\|_{C([-h_0,h_0])} + \|v_0\|_{C([-h_0,h_0])} \leq \mu^{-1}M \min \left\{ \min_{x \in [-h_0,h_0]} \phi_1(x), \min_{x \in [-h_0,h_0]} \phi_2(x) \right\}
\end{align}
provided that
 $0 < \mu \leq \underline{\mu}$ with
\begin{align*}
\underline{\mu} := &M \min\left\{\min_{x \in [-h_0,h_0]} \phi_1(x), \min_{x \in [-h_0,h_0]} \phi_2(x) \right\}\left( \|u_0\|_{C([-h_0,h_0])} + \|v_0\|_{C([-h_0,h_0])}  \right)^{-1}.
\end{align*}
Therefore, $0 < \mu \leq \underline{\mu}$ implies
\[
u_0(x) \leq \mu^{-1}M \phi_1 = \bar{u}(0,x), \quad v_0(x) \leq \mu^{-1}M \phi_2(x) = \bar{v}(0,x) \quad \text{for } x \in [-h_0, h_0].
\]
Summarizing the above discussions, we can use Lemma \ref{mm} to obtain
\[
u(t,x) \leq \bar{u}(t,x), ~ v(t,x) \leq \bar{v}(t,x), ~ \bar{g}(t) \leq g(t) \text{ and } h(t) \leq \bar{h}(t)
\]
for $t > 0$ and $x \in [g(t), h(t)]$. Therefore,
\[
h_\infty - g_\infty \leq \lim_{t \to \infty} (\bar{h}(t) - \bar{g}(t)) = 2h_1 < +\infty,
\]
and hence vanishing happens. \qed

\begin{remark}\label{88}
\textnormal{From the proof of Lemma \ref{e2}, we can see that for any fixed $\mu > 0$, a sufficiently small initial function pair $(u_0, v_0)$ still guarantees that \eqref{e3} holds. In other words, if $1<\mathcal{R}_0<\left(1+\frac{d_1}{a}\right)\left(1+\frac{d_2}{b}\right)$ and $h_0<L^*$, then for any fixed $\mu > 0$ and sufficiently small initial datum $(u_0, v_0)$,  vanishing always occurs.}
\end{remark}

\begin{lemma}\label{e6}
Suppose that $1 < \mathcal{R}_0 < (1 + \frac{d_1}{a})(1 + \frac{d_2}{b})$ and $h_0 < L^*$ hold. Then there exists $\bar{\mu}> 0$ depending on $(u_0, v_0)$ such that spreading happens for \eqref{A} if $\mu > \bar{\mu}$.
\end{lemma}
\noindent{\bf{Proof}}. Suppose for contradiction that there exists an admissible initial function pair $(u_0, v_0)$ such that $h_\infty-g_\infty<+\infty$ for any $\mu > 0$. By Lemma \ref{d3}, we have $\lambda_p(g_\infty, h_\infty) \ge 0$. This indicates that $h_\infty - g_\infty \leq 2L^*$.
To emphasize the dependence of the solution of problem~\eqref{A} on \( \mu \), we denote it by \( (u_\mu, v_\mu, g_\mu, h_\mu) \). Thanks to Corollary~\ref{nn}, we know that \( u_\mu, v_\mu, -g_\mu \), and \( h_\mu \) are increasing with respect to \( \mu > 0 \).

Denote
\begin{align*}
h_{\mu,\infty} &:= \lim_{t \to \infty} h_\mu(t), \quad g_{\mu,\infty} := \lim_{t \to \infty} g_\mu (t).
\end{align*}
Clearly, both $-g_{\mu,\infty}$ and $h_{\mu,\infty}$ are nondecreasing in $\mu > 0$ and bounded. Therefore,
\begin{align*}
H_\infty &:= \lim_{\mu \to +\infty} h_{\mu,\infty} < +\infty, \quad G_\infty := \lim_{\mu \to +\infty} g_{\mu,\infty} > -\infty.
\end{align*}
By the conditions {\bf{(J)}} and {\bf{(W)}}, there exist constants $\epsilon_0 \in (0, h_0/2)$ and $\delta_0 > 0$ such that $J_1(x-y) > \delta_0$ for $|x-y| \leq \epsilon_0$ and $\int_{\epsilon_0}^{2h_0} W(x) \, dx > \delta_0$. Then for such $\epsilon_0$, there exist $\mu_0, t_0 > 0$ such that for all $\mu \geq \mu_0$, $t \geq t_0$, we have
\begin{align*}
H_\infty - h_\mu(t) &< \frac{\epsilon_0}{4}.
\end{align*}
Therefore, for any $\mu \geq \mu_0$, $t \geq t_0$, we have
{\small \begin{align*}
&h_{\mu,\infty} - h_\mu(t_0)\\
=&\ \mu \int_{t_0}^{+\infty} \left( \int_{g_\mu(\tau)}^{h_\mu(\tau)} \int^{+\infty}_{h_\mu(\tau)} J_1(x-y) u_\mu(\tau,x) \rd y  \rd x + \rho \int_{g_\mu(\tau)}^{h_\mu(\tau)} v_\mu(\tau,x) W(h_\mu(\tau) - x)\rd x \right) \rd \tau \\
\geq&\ \mu \int_{t_0}^{t_0 + 1} \left( \int_{g_{\mu_0}(\tau)}^{h_{\mu_0}(\tau)}\int_{h_{\mu_0}(\tau) + \epsilon_0/4}^{+\infty} J_1(x-y) u_{\mu_0}(\tau,x)  \rd y  \rd x \right.\\
&\left.+ \rho \int^{h_{\mu_0}(\tau) }_{g_{\mu_0}(\tau)} v_{\mu_0}(\tau,x) W(h_\mu(\tau) - x)\rd x\right)  \rd \tau \\
\geq&\ \mu \int_{t_0}^{t_0 + 1} \left( \int_{h_{\mu_0}(\tau) - \epsilon_0/2}^{h_{\mu_0}(\tau)}\int_{h_{\mu_0}(\tau) + \epsilon_0/4}^{h_{\mu_0}(\tau) + \epsilon_0/2} J_1(x-y)  u_{\mu_0}(\tau,x)  \rd y \rd x \right.\\
&\left.+ \rho \inf_{x\in[g_{\mu_0}(\tau), h_{\mu_0}(\tau)]} v_{\mu_0}(\tau,x)  \int_{g_{\mu_0}(\tau)}^{h_{\mu_0}(\tau)} W(h_\mu(\tau)-x) \rd x \right) \rd \tau\\
\ge&\ \mu \int_{t_0}^{t_0 + 1}\left(\frac{1}{4}\epsilon_0 \delta_0 \int_{h_{\mu_0}(\tau)-\epsilon_0/2}^{h_{\mu_0}(\tau)}u_{\mu_0}(\tau,x)\rd x+\rho
\inf_{x\in[g_{\mu_0}(\tau),h_{\mu_0}(\tau)]}v_{\mu_0}(\tau,x)\int_{h_\mu(\tau)-
h_{\mu_0}(\tau)}^{h_\mu(\tau)-g_{\mu_0}(\tau)}W(x)\rd x\right)\rd \tau\\
\geq&\ \mu \int_{t_0}^{t_0 + 1} \left( \frac{1}{4} \epsilon_0 \delta_0 \int_{h_{\mu_0}(\tau) - \epsilon_0/2}^{h_{\mu_0}(\tau)} u_{\mu_0}(\tau, x)  \rd x + \rho \inf_{x \in [g_{\mu_0}(\tau), h_{\mu_0}(\tau)]} v_{\mu_0}(\tau, x) \int_{\epsilon_0}^{2h_0} W(x)  \rd x \right) \rd \tau \\
\geq&\ \mu \int_{t_0}^{t_0 + 1} \left( \frac{1}{4} \epsilon_0 \delta_0 \int_{h_{\mu_0}(\tau) - \epsilon_0/2}^{h_{\mu_0}(\tau)} u_{\mu_0}(\tau, x)  \rd x + \rho \delta_0 \inf_{x \in [g_{\mu_0}(\tau), h_{\mu_0}(\tau)]} v_{\mu_0}(\tau, x) \right) \rd \tau\\
\to & \ +\infty \mbox{ as } \mu\to\infty.
\end{align*}}
 This contradiction implies the existence of  $\overline{\mu} > 0$ such that spreading happens when $\mu > \overline{\mu}$. \qed

Below is a sharp criteria in terms of $\mu$ for the spreading-vanishing dichotomy.
\begin{lemma}\label{e7}
Suppose that $1 < \mathcal{R}_0 < \left(1 + \frac{d_1}{a}\right)\left(1 + \frac{d_2}{b}\right)$ and $h_0 < L^*$ hold. Then there exists $\mu^* > 0$ depending on $(u_0, v_0)$ such that  vanishing occurs when $0 < \mu \leq \mu^*$ and spreading occurs when $\mu > \mu^*$.
\end{lemma}
\noindent{\bf{Proof.}} This lemma can be proved  as in \cite{fb4}, and we omit the details. \qed

\vspace{1em}

Next we determine spreading or vanishing of system~\eqref{A} by a parameterization of  the initial value $(u_0,v_0)$. So    we assume that  $(u_0,v_0)=\sigma(\psi_1,\psi_2)$, where $\sigma>0$ is regarded as a parameter and $(\psi_1,\psi_2)$ is fixed and satisfies \eqref{B}.

\begin{lemma}\label{zzz}
Suppose that $1 < \mathcal{R}_0 < (1 + \frac{d_1}{a})(1 + \frac{d_2}{b})$, $h_0 < L^*$. If
\[
\mbox{either  ${\rm (i)}\ \ J_1(x)>0$ in $\mathbb{R}$ or \ ${\rm (ii)}\ \ W(x)>0 $ in $[0, 2L^*]$,}
\] then there exists $\bar{\sigma}> 0$ depending on $(\psi_1,\psi_2)$ such that spreading happens for \eqref{A} if $\sigma > \bar{\sigma}$.
\end{lemma}
\noindent{\bf{Proof}}. Arguing indirectly we assume that $h_\infty-g_\infty<+\infty$ for any $\sigma>0$. By Lemma \ref{d3} and Proposition \ref{c3}, we know that
\[
\mbox{$\lim\limits_{t \to \infty}\|u\|_{C([g(t),h(t)])}=\lim\limits_{t \to \infty}\|v\|_{C([g(t),h(t)])}=0$ and $h_\infty-g_\infty\leq 2L^*$.}
\]

If $J_1(x)>0$ in $\mathbb{R}$, then for any $t>0$ and $x\in(g(t),h(t))$,
$$W_{J_1}(h(t)-x)=\int_{h(t)-x}^{\infty}J(z) \rd z\geq\int^{\infty}_{2L^*}J(z)\rd z=:\varrho_1>0.$$
It follows that
$$\dfrac{1}{\mu}h'(t)\geq\int_{g(t)}^{h(t)}u(t,x)W_{J_1}(h(t)-x)
\rd x\geq\varrho_1\int_{g(t)}^{h(t)}u(t,x)\rd x.$$
Hence
$$\int_0^\infty\int_{g(s)}^{h(s)}u(s,x)\rd x\rd s\leq\frac
{h_\infty-h_0}{\mu\varrho_1}=:\varrho_2.$$
A straightforward calculation yields
\begin{align}\label{qe}
&\frac{d}{dt}\int_{g(t)}^{h(t)}u(t,x)\rd x\nonumber\\
&=\int_{g(t)}^{h(t)}u_t(t,x)\rd x+h'(t)u(t,x)\big|_{(t,
h(t))}+g'(t)u(t,x)\big|_{(t,g(t))}\nonumber\\
&=\int_{g(t)}^{h(t)}\left[ d_1 \int_{g(t)}^{h(t)} J_1(x-y) u(t,y) \rd y - d_1 u - a u + e v\right]\rd x\nonumber\\
&\geq \frac{d_1}{\mu}\left[\mu\int_{g(t)}^{h(t)}\int_{g(t)}^{h(t)}J_1(x-y) u(t,y) \rd y\rd x-\mu\int_{g(t)}^{h(t)}u(t,x)\rd x\right]-a\int_{g(t)}^{h(t)}u(t,x)\rd x.
\end{align}
Moreover, we have
\begin{align*}
-[h'(t)-g'(t)]=&\ \mu\lf(\int_{g(t)}^{h(t)}\int_{g(t)}^{h(t)}J_1(x-y) u(t,x) \rd y\rd x-\int_{g(t)}^{h(t)}u(t,x)\rd x\rr)\\
&-\mu\rho\lf(\int_{g(t)}^{h(t)}v(t,x)W(h(t)-x)\rd x+\int_{g(t)}^{h(t)}v(t,x)W(x-g(t))\rd x\rr)\\
\leq&\  \mu\lf(\int_{g(t)}^{h(t)}\int_{g(t)}^{h(t)}J_1(x-y) u(t,y) \rd y\rd x-\int_{g(t)}^{h(t)}u(t,x)\rd x\rr).
\end{align*}
Together with \eqref{qe}, we see
$$
\frac{d}{dt}\int_{g(t)}^{h(t)}u(t,x)\rd x\geq-\frac{d_1}{\mu}[h'(t)-g'(t)]-a\int_{g(t)}
^{h(t)}u(t,x)\rd x.
$$
Therefore, we have
$$
\int_{g(t)}^{h(t)}u(t,x)\rd x\geq\sigma\int_{-h_0}^{h_0}\psi_1(x)\rd x-2\frac{d_1}{\mu}L^*-
a\int_0^t\int_{g(s)}^{h(s)}u(s,x)\rd x\rd s.
$$
Summarizing the above discussions and letting $t \to \infty$ we obtain
$$
\sigma\leq\frac{1}{\int_{-h_0}^{h_0}\psi_1(x)\rd x}\left(a\varrho_2+2\frac{d_1}{\mu}L^*\right),
$$
which contradicts the unboundedness of $\sigma$. This proves the lemma when (i) holds.

If (ii) holds, namely $W(x)>0$ over $[0, 2L^*]$, then there exists $\tilde\varrho_1>0$ such that $W(x)\geq \tilde\varrho_1$ for $x\in [0, 2L^*]$. It follows that, for any $t>0$ and $x\in(g(t),h(t))$,
$$\dfrac{1}{\mu}h'(t)\geq \rho \int_{g(t)}^{h(t)}v(t,x)W(h(t)-x)
\rd x\geq \rho \tilde \varrho_1\int_{g(t)}^{h(t)}v(t,x)\rd x.$$
Hence
$$\int_0^\infty\int_{g(s)}^{h(s)}v(s,x)\rd x\rd s\leq\frac
{h_\infty-h_0}{\mu\rho\tilde \varrho_1}=:\tilde \varrho_2.$$
From this we can argue analogously as in case (i) above to deduce a contradiction.
The proof is now complete. \qed

\begin{lemma}\label{xxx} Under the conditions of Lemma \ref{zzz},
 there exists a unique $\sigma^*> 0$ such that  vanishing occurs when $0<\sigma\leq\sigma^*$ and spreading occurs when $\sigma > \sigma ^*$.
\end{lemma}
\noindent{\bf Proof.}  Denote by $(u_\sigma,v_\sigma,g_\sigma,h_\sigma)$ the unique solution of \eqref{A} with $(u_0,v_0)=\sigma(\psi_1,\psi_2)$, and define $h_{\sigma,\infty}:=\lim_{t \to \infty} h_\sigma(t)$, $g_{\sigma,\infty}:=\lim_{t \to \infty}g_\sigma(t)$.
We further define
$$
\sigma^*:=\sup\{\sigma:\sigma>0~\text{such that}~h_{\sigma,\infty}-g_{\sigma,\infty}<+\infty\}.
$$
According to Remark \ref{88} and Lemma \ref{zzz}, we have $0<\sigma^*<+\infty$. Similar to Corollary \ref{nn}, it follows from Lemma \ref{mm} that $u_\sigma,v_\sigma,-g_\sigma,h_\sigma$ are increasing in $\sigma>0$. Therefore,  by the definition of $\sigma^*$, we obtain
\begin{equation}\label{999}
h_{\sigma,\infty}-g_{\sigma,\infty}<+\infty~ \text{for all}~ \sigma \in (0,\sigma^*)~\text{and}~ h_{\sigma,\infty}-g_{\sigma,\infty}=+\infty~\text{for all}~ \sigma \in (\sigma^*,+\infty).
\end{equation}
It remains to show that $h_{\sigma^*,\infty}-g_{\sigma^*,\infty}<+\infty$. Otherwise, there exists $T>0$ such that $h_{\sigma^*}(T)-g_{\sigma^*}(T)>2L^*$. By the continuous dependence of the solution to \eqref{A} on the parameter $\sigma$, we can find a small $\varepsilon>0$ such that
\[
h_{\sigma,\infty}-g_{\sigma,\infty}>h_{\sigma}(T)-g_{\sigma}(T)>2L^*\mbox{ and }\sigma \in (\sigma^*-\varepsilon,\sigma^*).
\]
Combining this with Proposition~\ref{c3} and Lemma~\ref{d3}, we conclude that
\[
h_{\sigma,\infty}-g_{\sigma,\infty}=+\infty\mbox{ for }\sigma \in (\sigma^*-\varepsilon,\sigma^*),
\]
which contradicts \eqref{999}. Hence we have $h_{\sigma^*,\infty}-g_{\sigma^*,\infty}<+\infty$. The proof is finished.\qed

\vspace{1em}

It is easily seen that Theorem \ref{T13}
 follows directly from Lemmas \ref{d5}, \ref{666}, \ref{d9}, \ref{e1}, \ref{e7} and \ref{xxx}, while Theorem \ref{thm6} follows from Lemmas \ref{d5} and \ref{R0-1}.

\subsection{Effect of the diffusion rate}

In this subsection, we study the effect of the diffusion rate on the dynamical behaviour of disease transmission, and we always assume
\[
(d_1, d_2)=s(d_1^0, d_2^0)
\]
 with $(d_1^0, d_2^0)$ a fixed pair of positive numbers and $s>0$ regarded as a parameter.
  Let \( \lambda_p(s, 2h_0) \) denote the principal eigenvalue of \eqref{c4} with  \((d_1, d_2)=s(d_1^0, d_2^0),\ \  [L_1, L_2] = [-h_0, h_0] \).

By Proposition \ref{c9}, we see that $\lambda_p(s,2h_0) > 0$ for any $s > 0$ if $\mathcal{R}_0 \leq 1$. In the previous subsection, we have shown that when \( \mathcal{R}_0 \leq 1 \), the species \( u \) and \( v \) will eventually go extinct regardless of the value of  \( s > 0 \). Therefore, in what follows, we focus on the case \( \mathcal{R}_0 > 1 \), and aim to determine a sharp threshold value of \( s \) that determines the spreading or vanishing  of system~\eqref{A}, as stated in Theorem~\ref{T14}.

Recall from Proposition~\ref{c9} that
\[
\lim_{s \to 0} \lambda_p(s,2h_0) = \frac{1}{2} \left( \frac{a}{e} + \frac{b}{G'(0)} - \sqrt{\left( \frac{a}{e} - \frac{b}{G'(0)} \right)^2 + 4} \right) < 0, \quad \lim_{s \to \infty} \lambda_p(s,2h_0) = +\infty.
\]
Then there exists $d^* >0$ such that
\[
\lambda_p(s,2h_0) < 0 \text{ if } s < d^*,~\lambda_p(s,2h_0) = 0 \text{ if } s = d^* \text{ and } \lambda_p(s,2h_0) > 0 \text{ if } s > d^*.
\]

\begin{lemma} \label{e8}
Suppose that $\mathcal{R}_0>1$ and $s\leq d^*$ hold. Then spreading always happens for \eqref{A}.
\end{lemma}
\noindent{\bf Proof.} For any fixed $s\leq d^*$, we have $\lambda_p(s,2h_0)\leq 0$. Arguing indirectly we assume that $h_\infty-g_\infty<+\infty$. It then follows from Proposition \ref{c3} that
$$
\lambda_p(s,h_\infty-g_\infty)<\lambda_p(s,2h_0)\leq 0,
$$
which is a contradiction to Lemma \ref{d3}. This completes the proof.\qed

\begin{lemma}\label{e9}
Suppose that $\mathcal{R}_0 > 1$ and $s > d^*$ hold. Then there exists $\mu_0 > 0$ such that  vanishing happens for \eqref{A} if $0 < \mu \leq \mu_0$.
\end{lemma}
\noindent{\bf Proof.} For any fixed $s > d^*$, we have $\lambda_p(s, 2h_0) > 0$. Thanks to Proposition \ref{c3}, there exists $h_1>h_0$ but close to $h_0$ such that $\lambda_1:=\lambda_p(s, 2h_1) > 0$. We next consider the following problem:
\begin{equation}\label{g1}
\begin{cases}
u_{1t}=d_1 \displaystyle\int_{-h_1}^{h_1} J_1(x-y) u_1(t,y)  \rd y - d_1 u_1 - a u_1 + e v_1, & t>0, \, x\in[-h_1, h_1], \\
v_{1t}=d_2 \displaystyle\int_{-h_1}^{h_1} J_2(x-y) v_1(t,y)  \rd y - d_2 v_1 - b v_1 + G(u_1), & t>0, \, x\in[-h_1, h_1], \\
u_1(0,x)=u_0(x), \, v_1(0,x)=v_0(x), &x\in[-h_0, h_0],\\
u_1(t,x)=v_1(t,x)=0,&x\in[-h_1, -h_0)\cup(h_0,h_1].
\end{cases}
\end{equation}
Now, we use $(\tilde{u}_1(t,x), \tilde{v}_1(t,x))$ to denote the unique solution of $\eqref{g1}$. Let $(\phi_1, \phi_2)$ be a positive and normalized eigenfunction pair corresponding to $\lambda_1$ satisfying $\|\phi_1\|_\infty = \|\phi_2\|_\infty = 1$. Denote $\delta:=\min\{e\lambda_1,G'(0)\lambda_1\}$. Then, for any $c_1 > 0$, define\
\[
(\hat{u}_1(t,x), \hat{v}_1(t,x)) = (c_1 e^{-\frac{\delta}{2} t} \phi_1, c_1 e^{-\frac{\delta}{2} t} \phi_2).
\]
A direct computation yields
\begin{align*}
&d_1 \int_{-h_1}^{h_1} J_1(x-y) \hat{u}_1(t,y)  \rd y - d_1 \hat{u}_1(t,x) - a \hat{u}_1(t,x) + e \hat{v}_1(t,x) - \hat{u}_{1t}(t,x) \\
&= c_1 e^{-\frac{\delta}{2} t} \left( d_1 \int_{-h_1}^{h_1} J_1(x-y) \phi_1(y)  \rd y - d_1 \phi_1(x) - a \phi_1(x) + e \phi_2(x) + \frac{\delta}{2} \phi_1(x) \right) \\
&\leq c_1 e^{-\frac{\delta}{2} t} \left( -\frac{\delta}{2} \phi_1(x) \right) \\
&< 0.
\end{align*}
Similarly, we have
\begin{align*}
&d_2 \int_{-h_1}^{h_1} J_2(x-y) \hat{v}_1(t,y)  \rd y - d_2\hat{v}_1(t,x) - b \hat{v}_1(t,x) + G(\hat{u}_1(t,x)) - \hat{v}_{1t}(t,x) \\
&\leq c_1 e^{-\frac{\delta}{2} t} \left( d_2 \int_{-h_1}^{h_1} J_2(x-y) \phi_2(y)  \rd y - d_2 \phi_2(x) - b \phi_2(x) + G'(0) \phi_1(x) + \frac{\delta}{2} \phi_2(x) \right) \\
&\leq 0.
\end{align*}
Choosing $c_1>0$ suitably large such that $c_1 \phi_1 > u_0$ and $c_1 \phi_2 > v_0$ in $[-h_1, h_1]$, we can apply Lemma \ref{hh} to obtain
\begin{align}\label{g2}
\tilde{u}_1(t,x) &\leq \hat{u}_1(t,x) = c_1 e^{-\frac{\delta}{2} t} \phi_1 \leq c_1 e^{-\frac{\delta}{2} t}, \nonumber\\
\tilde{v}_1(t,x) &\leq \hat{v}_1(t,x) = c_1 e^{-\frac{\delta}{2} t} \phi_2 \leq c_1 e^{-\frac{\delta}{2} t}
\end{align}
for $t>0$ and $x \in [-h_1, h_1]$. Next, we define
\begin{align*}
\tilde{h}(t) &:= h_0 + 4 \mu h_1 c_1 \left( \rho\sup_{x\in[0, 2h_1]}W(x)+1  \right) \int_0^t e^{-\frac{\delta}{2} s}  \rd s \quad \text{and} \quad \tilde{g}(t) := -\tilde{h}(t) \quad \text{for } t \geq 0,
\end{align*}
where $0 < \mu \leq \mu_0 := \frac{\delta (h_1 - h_0)}{ 8h_1 c_1 \left( \rho\sup_{x\in[0, 2h_1]}W(x)+1  \right)}$. We are going to show that $(\tilde{u}, \tilde{v}, \tilde{g}, \tilde{h})$ is an upper solution of \eqref{A}. Thanks to the choice of $\mu$, it is easy to check that
\begin{align*}
\tilde{h}(t)
&= h_0 + 4 \mu h_1 c_1 \left( \rho\sup_{x\in[0, 2h_1]}W(x)+1  \right) \frac{2}{\delta} \left(1 - e^{-\frac{\delta}{2} t}\right) \\
&< h_0 + \frac{8}{\delta} \mu h_1 c_1 \left( \rho\sup_{x\in[0, 2h_1]}W(x)+1  \right) \leq h_1 \quad \text{for any } t \geq 0.
\end{align*}
Similarly, $\tilde{g}(t) > -h_1$ for any $t \geq 0$. Thus \eqref{g1} implies that
\begin{equation*}
\begin{cases}
\tilde{u}_{1t} \geq d_1 \displaystyle\int_{\tilde{g}(t)}^{\tilde{h}(t)} J_1(x-y) \tilde{u}_1(t,y)  \rd y - d_1 \tilde{u}_1 - a \tilde{u}_1 + e \tilde{v}_1, ~ &t > 0, ~ x \in [\tilde{g}(t), \tilde{h}(t)], \\
\tilde{v}_{1t} \geq d_2 \displaystyle\int_{\tilde{g}(t)}^{\tilde{h}(t)} J_2(x-y) \tilde{v}_1(t,y)  \rd y - d_2 \tilde{v}_1 - b \tilde{v}_1 + G(\tilde{u}_1), ~ &t > 0, ~ x \in [\tilde{g}(t), \tilde{h}(t)].
\end{cases}
\end{equation*}
On the other hand,
\begin{align*}
\tilde{h}'(t) &= 4 \mu h_1 c_1 \left( \rho\sup_{x\in[0, 2h_1]}W(x)+1  \right) e^{-\frac{\delta}{2} t} \\
&\geq \mu \int_{\tilde{g}(t)}^{\tilde{h}(t)} \int_{\tilde{h}(t)}^{+\infty} J_1(x-y) \tilde{u}_1(t,x)  \rd y  \rd x + \mu \rho \int_{\tilde{g}(t)}^{\tilde{h}(t)} \tilde{v}_1(t,x) W(\tilde{h}(t)-x)  \rd x
\end{align*}
and
\begin{align*}
\tilde{g}'(t) &= -4 \mu h_1 c_1 \left(\rho\sup_{x\in[0,2h_1]} W(x) + 1\right) e^{-\frac{\delta}{2} t} \\
&\leq -\mu\int^{\tilde{h}(t)}_{\tilde{g}(t)} \int^{\tilde{g}(t)}_{-\infty} J_1(x-y) \tilde{u}_1(t,x)  \rd y  \rd x - \mu \rho \int_{\tilde{g}(t)}^{\tilde{h}(t)} \tilde{v}_1(t,x) W(x-\tilde{g}(t))  \rd x.
\end{align*}
Hence $(\tilde{u}, \tilde{v}, \tilde{g}, \tilde{h})$ is an upper solution of \eqref{A}, and it follows from
 Lemma \ref{mm} that
\begin{align*}
u(t,x) \leq \tilde{u}_1(t,x), ~ v(t,x) \leq \tilde{v}_1(t,x), ~ g(t) \geq \tilde{g}(t) \text{ and } h(t) \leq \tilde{h}(t) \quad\text{for } t \geq 0, ~ x \in [g(t), h(t)].
\end{align*}
Therefore,
$$
h_\infty-g_\infty\leq \lim_{t\rightarrow\infty}(\tilde{h}(t)-\tilde{g}(t))\leq2h_1<+\infty.
$$
The proof is completed. \qed

\begin{lemma}\label{313}
Suppose that $ \mathcal{R}_0 >1$ and $s> d^*$. Then there exists $\mu_0^* > 0$ such that  vanishing occurs when $0 < \mu \leq \mu_0^*$ and spreading occurs when $\mu > \mu_0^*$.
\end{lemma}
\noindent{\bf{Proof.}} The desired conclusion can be obtained by applying a similar argument as in the proofs of Lemmas~\ref{e6} and~\ref{e7}; the details are omitted.\qed

\vspace{1em}
Theorem \ref{T14} clearly follows directly  from Lemmas \ref{e8} and \ref{313}.


\section*{Acknowledgments}
\noindent

Y. Chen was partially supported by a scholarship from the China Scholarship Council (2024 06180049);  Y. Du was partially supported by the Australian Research Council; W.-T. Li was partially supported by NSF of China (12531008, 12271226), and R. Wang was partially supported by NSF of China (12401258) and the Fundamental Research Funds for the Central Universities (lzujbky-2024-pd10).


\end{document}